\documentclass{amsart}
\usepackage[usenames,dvipsnames]{xcolor}
\usepackage[mathscr]{euscript}
\usepackage{comment}
\usepackage{lineno}
\usepackage{hyperref}
\usepackage{diagbox}
\usepackage{subcaption, wrapfig}
\usepackage{amsmath, amsfonts, amssymb, amsthm, geometry, graphics, graphicx, multicol, multirow, enumerate, float, floatflt, fullpage, tikz, outlines, url, tabularx, tikz-cd, array,pgfplots,url}
\usetikzlibrary{fillbetween}
\usepackage{placeins}
\usepackage[utf8]{inputenc}
\usetikzlibrary{shapes.multipart, matrix, positioning, arrows,cd, decorations.pathmorphing, decorations.pathreplacing}
\usepackage{wasysym} 
\usepackage{mathdots} 
\usepackage{scalerel} 
\newlength\bshft 
\def\fakebold#1{\ThisStyle{\ooalign{$\SavedStyle#1$\cr%
  \kern-\bshft$\SavedStyle#1$\cr%
  \kern\bshft$\SavedStyle#1$}}}

\usetikzlibrary[topaths] 
\usetikzlibrary{graphs,graphs.standard,shapes.geometric}
\usetikzlibrary{decorations.pathreplacing,decorations.markings}
\usetikzlibrary{cd}

\tikzset{
  on each segment/.style={
    decorate,
    decoration={
      show path construction,
      moveto code={},
      lineto code={
        \path [#1]
        (\tikzinputsegmentfirst) -- (\tikzinputsegmentlast);
      },
      curveto code={
        \path [#1] (\tikzinputsegmentfirst)
        .. controls
        (\tikzinputsegmentsupporta) and (\tikzinputsegmentsupportb)
        ..
        (\tikzinputsegmentlast);
      },
      closepath code={
        \path [#1]
        (\tikzinputsegmentfirst) -- (\tikzinputsegmentlast);
      },
    },
  },
  mid arrow/.style={postaction={decorate,decoration={
        markings,
        mark=at position .5 with {\arrow[#1]{stealth}}
      }}},
}

\newtheorem{thm}{Theorem}[section]
\newtheorem{cor}[thm]{Corollary}

\newtheorem{prop}[thm]{Proposition}
\newtheorem{lem}[thm]{Lemma}
\newtheorem{conj}[thm]{Conjecture}

\newtheorem*{theorem*}{Theorem \ref{TorusLinksOnly2Torsion}}
\newtheorem*{theorem'}{Theorem \ref{nooddtorsion}}
\newtheorem*{theorem"}{Theorem \ref{GeneralTheorem2}}

\theoremstyle{definition}
\newtheorem{defn}[thm]{Definition}
\newtheorem{example}[thm]{Example}

\newcommand\restr[2]{{
  \left.\kern-\nulldelimiterspace 
  #1 
  \vphantom{\big|} 
  \right|_{#2} 
  }}
\newcommand{\R}{\mathbb{R}}
\newcommand{\Z}{\mathbb{Z}}

\newcommand{\I}{\mathcal{I}}

\newcommand{\ob}[1]{\overline{#1}}

\newcommand{\im}{\textnormal{Im}\,}

\newcolumntype{P}[1]{>{\centering\arraybackslash}p{#1}} 
\definecolor{forest}{rgb}{0.03, 0.47, 0.19}

\graphicspath{ {./_figures/} }

\newtheorem*{thm'}{Theorem \ref{allEMHTorsion}}
\newtheorem*{thm*}{Theorem \ref{posetTor}}
\newtheorem*{cor*}{Corollary \ref{Euler_diag}}
\newtheorem*{thm**}{Theorem \ref{CW_thm}}

\title{Torsion in magnitude homology theories}

\thanks{RS partially supported by the  NSF Grant DMS 1854705.}

\author{Patrick Martin II}
\address{Department of Mathematics\\
North Carolina State University\\
Raleigh, NC}
\email{pmmarti6@ncsu.edu}

\author{Radmila Sazdanovi\'{c}}
\address{Department of Mathematics\\
North Carolina State University\\
Raleigh, NC}
\email{rsazdanovic@math.ncsu.edu}

\begin{document}
\maketitle

\begin{abstract} 
In this article, we analyze the structure and relationships between magnitude homology and Eulerian magnitude homology of finite graphs. Building on the work of Kaneta and Yoshinaga, Sazdanovic and Summers, and Asao and Izumihara, we provide two proofs of the existence of torsion in Eulerian magnitude homology, offer insights into the types and orders of torsion, and present explicit computations for various classes of graphs.
\end{abstract}

\section{Introduction}\label{intro}
In 2006 Tom Leinster introduced the notion of an Euler characteristic for a finite category \cite{leinster2008euler}. Rooted in Rota's theory for posets \cite{rota1964foundations, stanley2011enumerative}, this Euler characteristic exhibits an inclusion-exclusion formula and other cardinality-like properties, which inspired its new name: magnitude \cite{Mag_roundup}.
Under the correct conditions, magnitude was shown to be the same thing as the Poincaré polynomial for posets and hyperplane arrangements \cite{asao2023magnitude}, further connecting it to theories in combinatorics. Magnitude has also recently found applications for edge detection in images and can potentially be used as a tool in machine learning research \cite{Magnitude_of_Vector_Images}.

Viewing graphs as metric spaces with the path-distance allows for defining magnitude for graphs \cite{leinster2019magnitude}: a formal power series over $\mathbb{Z}.$  Magnitude, like many other graph invariants such as the chromatic and Tutte polynomials has been categorified using Khovanov's framework \cite{khovanov1999categorification} first used to categorify the Jones polynomial. Therefore, the graded Euler characteristic of magnitude homology \cite{hepworth2015categorifying} defined by Hepworth and Willerton is the magnitude of a graph.  
 
 Properties of magnitude lift to magnitude homology e.g., the inclusion-exclusion formula for magnitude lifts to a Mayer-Vietoris sequence in magnitude homology. Magnitude and magnitude homology are invariant under graph operations such as Whitney and sycamore twists \cite{roff2022magnitude}. However, magnitude homology is a strictly stronger invariant than magnitude, as first shown by Gu \cite{gu2018graph} who discovered graphs with equal  magnitude but contained distinct magnitude homology groups. The question for the existence of torsion in magnitude homology was posed by Hepworth and Willerton \cite{hepworth2015categorifying} and was answered affirmatively first by Kaneta and Yoshinaga \cite{Kaneta_Yoshinaga}. Their result focused on $\mathbb{Z}_2$ torsion arising from relations with triangulations of $\mathbb{R}P^2.$ Extending their approach to generalized lens spaces, Sazdanovic and Summers \cite{sazdanovic2021torsion} show that any finitely generated abelian group may appear
as a subgroup of the magnitude homology of a graph, and, in particular, that torsion
of a given prime order can appear in the magnitude homology of infintely many graphs. Caputi and Collari \cite{caputi2024finite} establish relations between the existence of high-order torsion types and  complexity of the graph determined by its genus, and that both torsion types and the ranks of the magnitude of any given graph are bounded. Asao and Hiraoka \cite{asao2024girth} further relate the girth of the graph with magnitude homology, showing that the increase in girth leads to the thinning of magnitude homology, eventually leaving just a single diagonal \cite{tajima2021magnitude}. 

In 2024 Giusti and Menara  \cite{giusti2024eulerian} defined Eulerian Magnitude by removing the redundancy of magnitude chains between consecutive groups by prohibiting repetitions of vertices and its complement, discriminant magnitude homology. The Eulerian magnitude chain complex is a subcomplex of magnitude, so strictly smaller in size and more computable \cite{menara2024computing} and providing new pathways for relating structural properties of graphs and magnitude. Structure, torsion, and relations between these magnitude theories are explored in further detail in \cite{menara2024torsion, asao2024girth}. 

The main goal of this work is to analyze torsion, type and order, in magnitude homology theories and provide further insights into relations between these theories \cite{giusti2024eulerian, asao2024girth}. 
In particular,  we extend the Sazdanovic and Summers result on the existence of torsion to Eulerian magnitude homology, showing that any finitely generated group can be a subgroup of Eulerian magnitude.

\begin{thm'}
    Given a finitely generated abelian group
$A=\Z^r\oplus\Z_{p_1^{r_1}}\oplus\cdots\oplus\Z_{p_n^{r_n}} $
    there exists a graph $G$ and positive integers $k$ and $\ell$ such that 
    $A\trianglelefteq EMH_{k,\ell}(G).$
\end{thm'}

In addition, we provide independent proof of existence of torsion in Eulerian magnitude based on Asao-Izumihara approach \cite{asao2020geometric}. We prove that the existence of torsion in the order complex of some poset implies the existence of the same torsion in  the magnitude and Eulerian magnitude homologies of the graph  representing
the Hasse diagram of that poset. 

\begin{thm*} 
  For a ranked poset $P$, the following holds for $\hat0,\hat1\in V(\mathcal{G}(\widehat{P}))$: 
    \[MH_{\ast+2,\text{rk}(\widehat{P})}(\hat0,\hat1)\cong H_\ast(\Delta(P))\cong EMH_{\ast+2,\text{rk}(\widehat{P})}(\hat0,\hat1).\]
\end{thm*}

So far, all known torsion groups are  common to the magnitude and Eulerian magnitude, derived using Kaneta and Yoshinaga approach based on  triangulations of manifolds.  We show that there exist an infinite family of ranked posets with 2-torsion in the homology of their order complexes which are not face posets, and give the explicit equivalence class which generates this torsion subgroup. This result turns out to be a special case of the more general result in Theorem \ref{CW_thm}, which we use to provide explicit examples of graphs with $3$-torsion and $5$-torsion in their (Eulerian) magnitude homology.

\begin{thm**}
    Let $K$ be a finite dimensional regular CW complex with finitely many cells in each dimension. Then there exists an isomorphism 
    \[EMH_{k+2,\text{rk}(\widehat{\mathcal{F}}(K))}(\hat0,\hat1)\cong H_k(\Delta(\mathcal{F}(K)))\cong MH_{k+2,\text{rk}(\widehat{\mathcal{F}}(K))}(\hat0,\hat1)\]
    where $\hat0,\hat1\in V(\mathcal{G}(K))$ are the appropriate vertices corresponding to $\hat0,\hat1\in\widehat{\mathcal{F}}(K)$.
\end{thm**}  

Finally, we explore the relationship between Eulerian,  discriminant and magnitude homology, and give explicit computations for the ranks of the Eulerian and discriminant magnitude homology groups for some classes of trees, as well as the Eulerian magnitude homology for complete graphs. We also give an important result about diagonality in Eulerian magnitude homology, which is a corollary of another result about the `maximum Eulerian magnitude homology group' for a given graph.

\begin{cor*}
    A graph $G$ has diagonal Eulerian magnitude homology if and only if $G$ is a complete graph.
\end{cor*}

This paper is organized as follows. In Section \ref{back}, we give a construction of magnitude and Eulerian magnitude homology. In Sections \ref{tor} and \ref{Asa0Izumihara}, we provide two independent proofs of the existence of any finitely generated abelian group as a subgroup of the Eulerian magnitude homology. In Section \ref{noTriang} we provide an infinite family of examples with $\mathbb{Z}_2$ torsion on their magnitude that is independent of Kaneta-Yoshinaga and other currently known methods for obtaining torsion. In Section \ref{CW} we show that the Kaneta-Yoshinaga construction can be strengthened to include a wider class of objects which we then use to provide explicit examples of graphs with torsion in their magnitude homologies. In Section \ref{meaning} we provide further insight into the relation between magnitude homology theories and explicit computations for special families of graphs. Finally, we provide some conjectures and potential ideas for future research in Section \ref{Conj}.

\section{Background}\label{back}
In the course of this paper, given a graph $G$, $V(G)$ always represents the vertices of $G$ and $E(G)$ always represents the edges of $G$, and we require that both $|V(G)|$ and $|E(G)|$ be finite. A trail in $G$ from $a$ to $b$ with $a,b\in V(G)$ is an ordered tuple of vertices $(a,v_1,...,v_m,b)$, $v_i\in V(G)$, such that $\{a,v_1\},\{v_{i},v_{i+1}\},\{v_m,b\}\in E(G)$ for all $1\leq i\leq m-1$. Then the distance between any two vertices $a,b\in V(G)$ is 
    \[d(a,b):=\min\{m+1:(a,v_1,...,v_m,b)\text{ is a trail in }G\}\]
In other words, the distance between any two vertices is the shortest amount of edges needed to be able to connect the vertex $a$ to the vertex $b$. By convention, $d(a,b)=\infty$ if a trail from $a$ to $b$ cannot be found in $G$, and $d(a,a)=0$. 
\begin{defn}
    The \textbf{similarity matrix} of a graph $G$ is the matrix $\zeta_G$ indexed by the vertices of $G$ such that the $(a,b)$-entry in $\zeta_G$ is $\zeta_G(a,b)=q^{d(a,b)}$. 
\end{defn}
Note that evaluating $q=0$ yields the identity matrix implying that $\det (\zeta_G)$ has constant term $1$ and is thus invertible over $\Z[\![q]\!]$. Hence, $\zeta_G$ is always invertible for any graph $G$  \cite{leinster2019magnitude}.
\begin{defn}
    The \textbf{Möbius matrix} of a graph $G$ is the matrix $\zeta_G^{-1}$. The \textbf{magnitude} $\#G$ of a graph $G$ is the sum of all the entries of the Möbius matrix of $G$:
        \[\#G:=\sum_{a,b\in V(G)}\zeta_G^{-1}(a,b).\]
\end{defn}
The magnitude of a graph is a power series over $\Z$, and although it would seem that this would make it hard to compute, Leinster provides a simple formula \cite{leinster2019magnitude}. 
\begin{prop}\label{magnitude_graph_equality}
    Magnitude $\#G$ of  any graph $G$ is equal to 
        \[\#G=\sum_{k=0}^\infty(-1^k)\sum_{\substack{0\leq i<k \\x_i\neq x_{i+1},x_{k-1\neq x_k}\\ x_i,x_k\in V(G)}}q^{d(x_0,x_1)+d(x_1,x_2)+\cdots+d(x_{k-1},x_k)}.\]
\end{prop}

Next we provide definitions of basic objects used in the categorification of magnitude by Hepworth and Willerton \cite{hepworth2015categorifying}. 

\begin{defn}
    A \textbf{$k$-path} in $G$ is a $(k+1)$-tuple of vertices $(v_0,...,v_k)\in V(G)^{k+1}$ such that $v_i\neq v_{i+1}$ and $d(v_i,v_{i+1})<\infty$ for $0\leq i\leq k-1$. The \textbf{length of a $k$-path} is defined to be the sum of lengths of consecutive vertices:
        \[\text{len}(v_0,...,v_k):=\sum_{i=0}^{k-1}d(v_i,v_{i+1}).\]
\end{defn}

\begin{defn}
    Let $P_k(G)$ be the set of all $k$-paths in $G$. The \textbf{$(k,\ell)$-magnitude chain group} of a graph $G$ is generated by all $k$-paths of length $\ell$:   
        \[MC_{k,\ell}(G):=\Z\langle v\in P_k(G):\text{len}(v)=\ell\rangle.\]
\end{defn}

\begin{prop}[Lemma 11 \cite{hepworth2015categorifying}]
    For $v\in P_k(G)$ let $v^{\setminus i}\in P_{k-1}(G)$ be the $(k-1)$-path with the $i$-th vertex in $v$ removed. Define the map $\partial_{k,\ell}^i:MC_{k,\ell}(G)\to MC_{k-1,\ell}(G)$ on the generators
        \[\partial_{k,\ell}^i(v):=\begin{cases}
            v^{\setminus i} & \text{len}(v^{\setminus i})=\ell \\ 
            0 & \text{otherwise}.
        \end{cases}\]
    and then define the map $\partial_{k,\ell}:MC_{k,\ell}(G)\to MC_{k-1,\ell}(G)$ on the generators
        \[\partial_{k,\ell}(v):=\sum_{i=1}^{k-1}(-1)^i\partial_{k,\ell}^i(v).\]
    Then $\partial_{k,\ell}\partial_{k-1,\ell}(v)=0$.
\end{prop}

\begin{defn}
    The magnitude homology is the bigraded homology of the magnitude chain complex $(MC_{k,\ell}(G), \partial_{k,\ell})$ of a graph $G.$ More precisely,  \textbf{$(k,\ell)$-magnitude homology group} of a graph $G$ is 
        \[MH_{k,\ell}(G):=\ker\partial_{k,\ell}/\im\partial_{k+1,\ell}.\]
\end{defn}

Proposition \ref{magnitude_graph_equality} is used to show that  graph magnitude homology is indeed a categorification of the magnitude of a graph, result first proven by 
Hepworth and Willerton.

\begin{prop}[Theorem 8 \cite{hepworth2015categorifying}]
    Magnitude of a graph $G$ is equal to the Euler characteristic of Magnitude homology: 
    \[\#G=\sum_{k,\ell\geq0}(-1)^k\text{rank}(MH_{k,\ell}(G))q^\ell.\]
\end{prop}

Among the discussion of magnitude's cardinality-like properties by Leinster in \cite{leinster2019magnitude}, he reveals that for a ring $R$ the only $R$-valued graph invariant that satisfies inclusion-exclusion in order. The magnitude of a graph then, cannot have an inclusion-exclusion formula for graphs in general. However, under fairly plausible conditions on the graph its magnitude can exhibit an inclusion-exclusion formula which lifts to a Mayer-Vietoris sequence in its magnitude homology.

\begin{defn}
    A subgraph $H$ of $G$ is convex in $G$ if $d_H(a,b)=d_G(a,b)$ for all $a,b\in H$, where $d_H$ and $d_G$ are the distance functions with respect to each graph.
\end{defn}

\begin{defn}
    If $H$ is a subgraph of $G$, and $V_H(G)$ are those vertices $v\in V(G)$ such that $d(v,w)<\infty$ for all $w\in V(H)$, then $G$ projects onto $H$ if for each $x\in V_H(G)$ there exists a $\pi(x)\in V(H)$ such that $d(x,h)=d(x,\pi(x))+d(\pi(x),h)$. 
\end{defn}

\begin{lem}
   [Theorem 4.9 \cite{leinster2019magnitude}] If $H$ and $K$ are subgraphs of $G$ such that $G=H\cup K$, $H\cap K$ is convex in $G$, and either $H$ or $K$ projects onto $H\cap K$ then $\#G=\#H+\#K-\#(H\cap K)$.
\end{lem}

\begin{thm}[Theorem 29 
    \cite{hepworth2015categorifying}] If $H$ and $K$ are subgraphs of $G$ such that $G=H\cup K$, $H\cap K$ is convex in $G$, and either $H$ or $K$ projects onto $H\cap K$, then magnitude homology satisfies a short exact sequence 
        \[0\to MH_{k,\ell}(H\cap K)\to MH_{k,\ell}(H)\oplus MH_{k,\ell}(K)\to MH_{k,\ell}(G)\to 0.\]
\end{thm}

Asao and Izumihara \cite{asao2020geometric} were the first to observe that magnitude homology for graphs has a natural direct sum decomposition. Given any two arbitrary vertices $a,b\in V(G)$, define $MC_{k,\ell}(a,b)$ to be the free abelian group generated by all the $k$-paths of length $\ell$ in $G$ which begin with $a$ and end with $b$. Formally
\[MC_{k,\ell}(a,b):=\Z\langle v=(a,v_1,...,v_{k-1},b):\text{len}(v)=\ell\rangle.\]
Note that the boundary operator will naturally send elements in $MC_{k,\ell}(a,b)$ to an element in $MC_{k-1,\ell}(a,b)$ so that these turn out to be subcomplexes of $MC_{k,\ell}(G)$. Furthermore 
\[MC_{k,\ell}(G)=\bigoplus_{a,b\in V(G)}MC_{k,\ell}(a,b)\]
which also yields the following decomposition:
\[MH_{k,\ell}(G)=\bigoplus_{a,b\in V(G)}MH_{k,\ell}(a,b).\]

To contrast that of ordianry magnitude homology, Eulerian magnitude homology is defined in the same way with the exception that no vertices are allowed to repeat in a $k$-path of $G$. Note that the boundary map for magnitude homology will send Eulerian paths to Eulerian paths, thus yielding an Eulerian magnitude homology theory. 
\begin{defn}
    A $k$-path $(v_0,...,v_k)$ in a graph $G$ is an \textbf{Euleiran $k$-path} if $v_i\neq v_j$ for any $0\leq i,j\leq k$, $i\neq j$. Let $ET_k(G)$ be the set of all Eulerian $k$-paths in $G$. Then the \textbf{$(k,\ell)$-Eulerian magnitude chain group} of $G$ is 
        \[EMC_{k,\ell}(G):=\Z\langle v\in ET_k(G):\text{len}(v)=\ell\rangle,\]
    and the \textbf{$(k,\ell)$-Eulerian magnitude homology group} of $G$ is 
        \[EMH_{k,\ell}(G):=\ker\partial_{k,\ell}\big|_{EMC_{k,\ell}(G)}\Big/\im\partial_{k+1,\ell}\big|_{EMC_{k+1,\ell}(G)}.\]
    Furthermore, the \textbf{$(k,\ell)$-discriminant magnitude chain group} is $DMC_{k,\ell}(G):=MC_{k,\ell}(G)/EMC_{k,\ell}(G)$ and likewise the \textbf{$(k,\ell)$-discriminant magnitude homology group} is the relative homology group $DMH_{k,\ell}(G)$.
\end{defn}
Since discriminant magnitude homology is a relative homology theory, we have the following long exact sequence in magnitude homology 
    \[\cdots\to EMH_{k,\ell}(G)\to MH_{k,\ell}(G)\to DMH_{k,\ell}(G)\to EMH_{k-1,\ell}(G)\to\cdots\]
This sequence will be the main tool for computing the discriminant magnitude homology of star trees $S_n$ and the complete graphs $K_n$. 

Eulerian magnitude homology also enjoys a direct sum decomposition which is described in the same way as magnitude homology with the exception that $a\neq b$. Formally
\[EMC_{k,\ell}(a,b):=\Z\langle v=(a,v_1,...,v_{k-1},b)\in ET_k(G):a\neq b,\,\text{len}(v)=\ell\rangle,\]
and furthermore we get the direct sum decompositions:
\[EMC_{k,\ell}(G)=\bigoplus_{\substack{a,b\in V(G)\\a\neq b}}EMC_{k,\ell}(a,b),\]
\[EMH_{k,\ell}(G)=\bigoplus_{\substack{a,b\in V(G)\\ a\neq b}}EMH_{k,\ell}(a,b).\]

Now we turn our attention to the combinatorics which will be involved in our conversation about magnitude homology. In the course of this paper we will always let $[n]=\{1,2,3,...,n\}$ and $[n]_0=\{0,1,2,...,n\}$. Given a set $S$ we denote the set of all subsets of $S$ of cardinality $k$ as $\binom{S}{k}$. We also follow Sagan's notation for the falling factorial \cite{sagan2020combinatorics}, which represents the number of permutations of $[n]$ of length $k$:
\[n\downarrow_k\,:=\binom{n}{k}k!=\frac{n!}{(n-k)!}.\]
A poset $(P,\leq)$ is a set $P$ with a relation $\leq$ on the elements of $P$ which is reflexive, transitive, and antisymmetric. For two elements $a,b\in P$ we say $a$ covers $b$ if $a\leq c$ and $c\leq b$ implies that $a=c$ and we write $a\lessdot b$. For two elements $a,b\in P$ we write $a<b$ if $a\leq b$ and $a\neq b$. A chain of length $n$ in a poset is a sequence of elements $(a_0<a_1<\cdots<a_n)$ in $P$. A chain $(a_0<\cdots<a_n)$ is contained within a chain $(b_0<\cdots<b_m)$ if $m>n$ and there exists a subsequence $(b_{i_0}<\cdots<b_{i_n})$ such that $b_{i_j}=a_j$ for $0\leq j\leq n$. A chain is maximal or saturated if there exists no other chain which contains it. A poset $P$ is ranked with $\text{rank}(P)=n$ if every maximal chain is of the same length $n$. A poset $P$ has a minimal element denoted by $\hat0$ if for all $p\in P$, $\hat0\leq p$. Likewise $P$ has a maximal element denoted by $\hat1$ if for all $p\in P$, $\hat1\geq p$. Sometimes it is convenient to add a minimal and maximal element to a poset which may not have either. Thus we define the poset $\widehat{P}$ to be the poset $P$ with the added minimal and maximal elements. We can also assign a topological space to a poset.

\begin{defn}
    The \textbf{order complex} $\Delta(P)$ of a poset $P$ is the chain complex where the $n$th chain group is generated by the chains of length $n$ in $P$:
    \[C_n(\Delta(P)):=\Z\langle(a_0<\cdots<a_n):a_i\in P,\,i\in[n]_0\rangle.\]
\end{defn}

Although simplices in the order complex $\Delta(P)$ are generally considered to be unordered sets, it makes no difference for us to keep the order of the elements in a chain of length $n$ in $P$. Thus we keep the above notation when referring to generators of the chain groups of $\Delta(P)$. 


\section{Torsion in Eulerian magnitude homology}\label{tor}

 Kaneta-Yoshinaga \cite{Kaneta_Yoshinaga} start with a 
face poset  $\mathcal{F}(M)$. Nexrtm they consider  $\widehat{\mathcal{F}}(M)$ the poset $\mathcal{F}(M)$ with a $\hat0$ and a $\hat1$ adjoined as the unique minimal and maximal elements. Then $\mathcal{G}(M)$ denotes the graph of the Hasse diagram of $\widehat{\mathcal{F}}(M)$. Although the face poset $\mathcal{F}(M)$ and the graph $\mathcal{G}(K)$ are uniquely determined by the triangulation of $M$, minimal triangulation is preferred and ensures that  general, $\mathcal{F}(M)$ and $\mathcal{G}(M)$ are uniquely determined in the context of this paper.  This construction is essential in analyzing magnitude homology because of the following embedding: 

\begin{prop}[Corollary 5.12 \cite{Kaneta_Yoshinaga}]
(1) Let $P$ be a ranked poset with $\text{rank}(\widehat{P})=\ell$. then there exists an embedding of abelian groups $H_{k-2}(\Delta(P))\hookrightarrow H_{k,\ell}(\mathcal{G}(\widehat{P}))$, where $\mathcal{G}(\widehat{P})$ is the graph representing the Hasse diagram of $\widehat{P}$.\\
     (2) Let $M$ be a compact smooth manifold. Then for $\ell=\text{rank}(\widehat{\mathcal{F}}(M))$, $H_{k-2}(M)\cong MH_{k,\ell}(\hat0,\hat1)$ where $\hat0$ and $\hat1$ are the appropriate corresponding vertices in $\mathcal{G}(M)$ of the minimal and maximal elements in $\widehat{\mathcal{F}}(M)$.

\end{prop}

This construction led to the discovery of a graph that exhibits torsion in its magnitude homology; see Figure \ref{KY_rp2}. Following the sequence of isomorphisms given by Kaneta and Yoshinaga, it becomes evident that the torsion in the corresponding graph is generated by Eulerian paths. This arises because chains in a poset cannot contain repeated elements. Under the Kaneta-Yoshinaga map, these chains are mapped to their corresponding path representations in the Hasse diagram, which must therefore be Eulerian. Initially, this seems to produce two distinct Kaneta-Yoshinaga maps: one into magnitude homology and another into Eulerian magnitude homology.

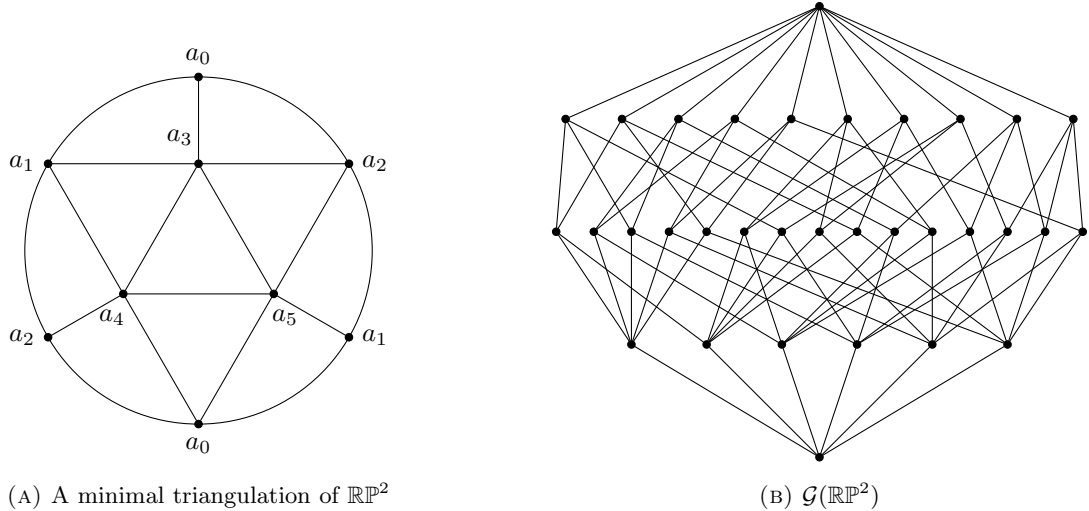
\begin{figure}[h]
    \centering
    \begin{subfigure}{.5\textwidth}
        \begin{center}
        \begin{tikzpicture}
        \node[circle, draw, fill=black, scale=0.3] (a) [label=west:$a_1$] at (-2,2) {};
        \node[circle, draw, fill=black, scale=0.3] (b) [label=east:$a_2$] at (2,2) {};
        \node[circle, draw, fill=black, scale=0.3] (c) [label=south:$a_0$] at (0,-1.464101) {};

        \draw[fill=none](0,0.845299) circle (2.309401);

        \node[circle, draw, fill=black, scale=0.3] (d) at (0,2) {};
        \node[label=$a_3$] at (-0.25,2) {}; 
        \node[circle, draw, fill=black, scale=0.3] (e)  at (1,0.267949) {};
        \node[label=$a_5$] at (1.15,-0.4) {};
        \node[circle, draw, fill=black, scale=0.3] (f)  at (-1,0.267949) {};
        \node[label=$a_4$] at (-1.15,-0.4) {};

        \node[circle, draw, fill=black, scale=0.3] (g) [label=north:$a_0$] at (0,3.1547005) {};
        \node[circle, draw, fill=black, scale=0.3] (h) [label=west:$a_2$] at (-2,-0.309) {};
        \node[circle, draw, fill=black, scale=0.3] (i) [label=east:$a_1$] at (2,-0.309) {};

        \draw (a) -- (b) -- (c) -- (a);
        \draw (d) -- (e) -- (f) -- (d);
        \draw (d) -- (g);
        \draw (e) -- (i);
        \draw (f) -- (h);
    \end{tikzpicture}
    \end{center}
        \caption{A minimal triangulation of $\mathbb{RP}^2$}
    \end{subfigure}%
    \begin{subfigure}{.5\textwidth}
        \begin{center}
        \begin{tikzpicture}[scale=0.5]
        \node[circle, draw, fill=black, scale=0.3] (0)  at (0,0) {};
        \node[circle, draw, fill=black, scale=0.3] (a_0)  at (-5,3) {};
        \node[circle, draw, fill=black, scale=0.3] (a_1)  at (-3,3) {};
        \node[circle, draw, fill=black, scale=0.3] (a_2)  at (-1,3) {};
        \node[circle, draw, fill=black, scale=0.3] (a_3)  at (1,3) {};
        \node[circle, draw, fill=black, scale=0.3] (a_4)  at (3,3) {};
        \node[circle, draw, fill=black, scale=0.3] (a_5)  at (5,3) {};

        \node[circle, draw, fill=black, scale=0.3] (01)  at (-7,6) {};
        \node[circle, draw, fill=black, scale=0.3] (02)  at (-6,6) {};
        \node[circle, draw, fill=black, scale=0.3] (03)  at (-5,6) {};
        \node[circle, draw, fill=black, scale=0.3] (04)  at (-4,6) {};
        \node[circle, draw, fill=black, scale=0.3] (05)  at (-3,6) {};
        \node[circle, draw, fill=black, scale=0.3] (12)  at (-2,6) {};
        \node[circle, draw, fill=black, scale=0.3] (13)  at (-1,6) {};
        \node[circle, draw, fill=black, scale=0.3] (14)  at (0,6) {};
        \node[circle, draw, fill=black, scale=0.3] (15)  at (1,6) {};
        \node[circle, draw, fill=black, scale=0.3] (23)  at (2,6) {};
        \node[circle, draw, fill=black, scale=0.3] (24)  at (3,6) {};
        \node[circle, draw, fill=black, scale=0.3] (25)  at (4,6) {};
        \node[circle, draw, fill=black, scale=0.3] (34)  at (5,6) {};
        \node[circle, draw, fill=black, scale=0.3] (35)  at (6,6) {};
        \node[circle, draw, fill=black, scale=0.3] (45)  at (7,6) {};

        \node[circle, draw, fill=black, scale=0.3] (013)  at (-6.75,9) {};
        \node[circle, draw, fill=black, scale=0.3] (015)  at (-5.25,9) {};
        \node[circle, draw, fill=black, scale=0.3] (023)  at (-3.75,9) {};
        \node[circle, draw, fill=black, scale=0.3] (024)  at (-2.25,9) {};
        \node[circle, draw, fill=black, scale=0.3] (045)  at (-.75,9) {};
        \node[circle, draw, fill=black, scale=0.3] (124)  at (.75,9) {};
        \node[circle, draw, fill=black, scale=0.3] (125)  at (2.25,9) {};
        \node[circle, draw, fill=black, scale=0.3] (134)  at (3.75,9) {};
        \node[circle, draw, fill=black, scale=0.3] (235)  at (5.25,9) {};
        \node[circle, draw, fill=black, scale=0.3] (345)  at (6.75,9) {};

        \node[circle, draw, fill=black, scale=0.3] (1)  at (0,12) {};

        \draw (0) -- (a_0);
        \draw (0) -- (a_1);
        \draw (0) -- (a_2);
        \draw (0) -- (a_3);
        \draw (0) -- (a_4);
        \draw (0) -- (a_5);

        \draw (a_0) -- (01) -- (a_1);
        \draw (a_0) -- (02) -- (a_2);
        \draw (a_0) -- (03) -- (a_3);
        \draw (a_0) -- (04) -- (a_4);
        \draw (a_0) -- (05) -- (a_5);
        \draw (a_1) -- (12) -- (a_2);
        \draw (a_1) -- (13) -- (a_3);
        \draw (a_1) -- (14) -- (a_4);
        \draw (a_1) -- (15) -- (a_5);
        \draw (a_2) -- (23) -- (a_3);
        \draw (a_2) -- (24) -- (a_4);
        \draw (a_2) -- (25) -- (a_5);
        \draw (a_3) -- (34) -- (a_4);
        \draw (a_3) -- (35) -- (a_5);
        \draw (a_4) -- (45) -- (a_5);

        \draw (01) -- (013);
        \draw (13) -- (013);
        \draw (03) -- (013);
        \draw (01) -- (015);
        \draw (15) -- (015);
        \draw (05) -- (015);
        \draw (02) -- (023);
        \draw (23) -- (023);
        \draw (03) -- (023);
        \draw (02) -- (024);
        \draw (24) -- (024);
        \draw (04) -- (024);
        \draw (04) -- (045);
        \draw (45) -- (045);
        \draw (05) -- (045);
        \draw (12) -- (124);
        \draw (24) -- (124);
        \draw (14) -- (124);
        \draw (12) -- (125);
        \draw (25) -- (125);
        \draw (15) -- (125);
        \draw (13) -- (134);
        \draw (34) -- (134);
        \draw (14) -- (134);
        \draw (23) -- (235);
        \draw (35) -- (235);
        \draw (25) -- (235);
        \draw (34) -- (345);
        \draw (45) -- (345);
        \draw (35) -- (345);

        \draw (013) -- (1);
        \draw (015) -- (1);
        \draw (023) -- (1);
        \draw (024) -- (1);
        \draw (045) -- (1);
        \draw (124) -- (1);
        \draw (125) -- (1);
        \draw (134) -- (1);
        \draw (235) -- (1);
        \draw (345) -- (1);
        
    \end{tikzpicture}
    \end{center}
    \caption{$\mathcal{G}(\mathbb{RP}^2)$}
    \end{subfigure}
    \caption{The Kenta-Yoshinaga construction for $\mathbb{RP}^2$}
    \label{KY_rp2}
\end{figure}

\begin{equation}\label{eq1}
     EMH_{k+2,\text{rk}(\widehat{P})}(\hat0,\hat1) \longleftarrow
    H_{k}(M) \longrightarrow
    MH_{k+2,\text{rk}(\widehat{P})}(\hat0,\hat1)
\end{equation}

However, this is not the case. In fact, for any such graph which represents the Hasse diagram of a ranked poset, the following chain groups are isomorphic. 
\begin{defn}
    Let $P$ be a poset. We define the graph representing the Hasse diagram of $P$ to be $\mathcal{G}(P)$. 
\end{defn}

Note that the elements of the set $P$ correspond to the vertices of $\mathcal{G}(P)$. For the remainder of this paper, the elements of $P$ shall be referred to as the vertices of $\mathcal{G}(P)$.

\begin{prop}
    If $P$ is a ranked poset, then for $\hat0,\hat1\in V(\mathcal{G}(\widehat{P}))$
\begin{equation}\label{EMCvsMC}
    EMC_{k,\text{rk}(\widehat{P})}(\hat0,\hat1)\cong MC_{k,\text{rk}(\widehat{P})}(\hat0,\hat1).
\end{equation}
\end{prop}
\begin{proof}
    Since $P$ is ranked, every maximal chain is saturated, and thus every path in $\widehat{G}$ of length $\text{rk}(\widehat{P})$ from $\hat0$ to $\hat1$ must be Eulerian. Since every such path is Eulerian, the above magnitude chain groups must be isomorphic. 
\end{proof}

Hence, the Kaneta-Yoshinaga approach embeds the homology group of a ranked poset into a direct summand of both magnitude homology and Eulerian magnitude homology. As a consequence of the results by Sazdanovic and Summers \cite{sazdanovic2021torsion} we get the following theorem.

\begin{thm}\label{allEMHTorsion}
    Given a finitely generated abelian group
$A=\Z^r\oplus\Z_{p_1^{r_1}}\oplus\cdots\oplus\Z_{p_n^{r_n}} $
    there exists a graph $G$ and positive integers $k$ and $\ell$ such that 
    $A\trianglelefteq EMH_{k,\ell}(G).$
\end{thm}

Therefore, all the known torsion which appears in magnitude homology also appears in Eulerian magnitude homology with the same equivalence class generating the torsion subgroups in both homology theories.

\section{Torsion via Asao-Izumihara construction}\label{Asa0Izumihara}

In 2020 Asao and Izumihara \cite{asao2020geometric} developed a geometric method of computing the magnitude homology of graphs which sheds a new light on the properties \cite{tajima2021magnitude} and was later modified to Eulerian Magnitude Homology \cite{giusti2024eulerian}. Since they were the first to present this construction, the complex is named after them. 

\begin{defn}
Given a graph $G$ and two vertices $a,b\in V(G)$, we can construct the simplicial complex
\[K_\ell(a,b):=\{\{(x_{i_1},i_1),...,(x_{i_k},i_k)\}\subseteq\{(x_1,1),...(x_m,m)\}\,|\,\text{len}(a,x_1,...,x_m,b)\leq\ell\}\]
and the subcomplex 
\[K'_\ell:=\{\{(x_{i_1},i_1),...,(x_{i_k},i_k)\}\in K_\ell(a,b)\,|\,\text{len}(a,x_{i_1},...,x_{i_k},b)\leq \ell-1\}.\]
The \textbf{Asao-Izumihara complex} on the graph $G$ for vertices $a,b\in V(G)$ is the quotient complex defined by the relative chain groups $C_n(K_\ell(a,b),K'_\ell(a,b))$.
\end{defn}

Informally, we are taking paths in $G$ from $a$ to $b$ of length less than or equal to $\ell$, then encoding the position of each landmark in the path in the second coordinate of each vertex in the Asao-Izumihara complex. Then to assure that we are only concerning ourselves with paths of length $\ell$ we quotient out by all paths of length less than $\ell$.

\begin{example}
    For the graph $G$ in Figure \ref{graph_G}, all the paths of length $4$ or less from vertex $1$ to vertex $5$ are as follows: $(1,2,1,5)$, $(1,3,1,5)$, $(1,3,2,5)$, $(1,3,4,5)$, $(1,4,3,5)$, $(1,5,3,5)$, $(1,2,3,5)$, $(1,2,5)$, $(1,3,5)$, $(1,4,5)$, and $(1,5)$. Thus we have 
    \[K_4(1,5)=\left\{
    \begin{matrix}
        \{(2,1),(1,2)\}, & \{(2,1),(3,2)\}, &\{(3,1),(1,2)\}, &\{(3,1),(2,2)\}, &\{(3,1),(4,2)\}, \\ 
        \{(4,1),(3,2)\}, &\{(5,1),(3,2)\}, &\{(2,1)\}, &\{(1,2)\},&\{(3,1)\}, \\ 
        \{(2,2)\}, &\{(4,2)\}, &\{(4,1)\}, &\{(3,2)\}, &\{(5,1)\}
    \end{matrix}
    \right\}\] 
    while $K'_4(1,5)=\{\{(2,1),(3,2)\},\{(2,1)\},\{(3,1)\},\{(3,2)\}\}$.
\end{example}

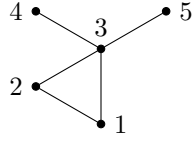
\begin{figure}
    \centering
    \begin{tikzpicture}
         \node[circle, draw, fill=black, scale=0.3] (1) [label=north:$3$] at (0,1) {};
         \node[circle, draw, fill=black, scale=0.3] (2) [label=west:$2$] at (-0.866,.5) {};
         \node[circle, draw, fill=black, scale=0.3] (3) [label=east:$1$] at (0,0) {};
         \node[circle, draw, fill=black, scale=0.3] (4) [label=west:$4$] at (-0.866,1.5) {};
         \node[circle, draw, fill=black, scale=0.3] (5) [label=east:$5$] at (0.866,1.5) {};
            
         \draw (1) -- (2) -- (3) -- (1);
         \draw (1) -- (4);
         \draw (1) -- (5);
     \end{tikzpicture}
    \caption{A graph $G$}
    \label{graph_G}
\end{figure}

Using the following result by Asao and Izumihara and the direct sum decomposition for magnitude homology, the Asao-Izumihara complex gives us a simplicial way to compute the magnitude homology of a graph.

\begin{prop}[Corollary 4.4 \cite{asao2020geometric}]\label{asao_izumi_complex}
Given a graph $G$ and vertices $a,b\in V(G)$, we have
\[H_k(K_\ell(a,b),K'_\ell(a,b))\cong MH_{k+2,\ell}(a,b).\]
\end{prop}

Next, based on extension of Asao and Izumihara approach the Eulerian magnitude theory \cite{menara2024computing} we 
get a statement for Eulerian magnitude homology analogous to  Proposition \ref{asao_izumi_complex}.

\begin{defn}
Given a graph $G$ and two vertices $a,b\in V(G)$ such that $a\neq b$, we can construct the simplicial complex
\[ET_{\leq\ell}(a,b):=\{\{(x_{i_1}i_1),...,(x_{i_k},i_k)\}\in K_\ell(a,b)\,|\,(a,x_{i_1},...,x_{i_k},b)\in ET_{k+1,\ell'}(G),\, \ell'\leq\ell\}\]
where $ET_{k,\ell}(G)$ is the set of Eulerian $k$-paths of length $\ell$ in $G$. The \textbf{Eulerian Asao-Izumihara complex} is the quotient complex defined by the relaive chain groups $C_n(ET_{\leq\ell}(a,b),ET_{\leq\ell-1}(a,b))$.
\end{defn}

\begin{prop}[Corollary 20 \cite{menara2024torsion}]
Given a graph $G$ and vertices $a,b\in V(G)$ with $a\neq b$, we have
\[H_k(ET_{\leq\ell}(a,b),ET_{\leq\ell-1}(a,b))\cong EMH_{k+2,\ell}(a,b).\]
\end{prop}

The geometric approach of Asao and Izumihara to both magnitude homology theories serves as the foundation for an alternative proof of Theorem \ref{allEMHTorsion}, independent of the embedding provided by Kaneta and Yoshinaga. However, the Kaneta-Yoshinaga approach remains relevant, as it is used to identify a poset with torsion in its order complex. 
\begin{prop}
    Let $P$ be a ranked poset. Then for $\hat0,\hat1\in V(\mathcal{G}(\widehat{P}))$ 
    \[K'_{\text{rk}(\widehat{P})}(\hat0,\hat1)=ET_{\leq\text{rk}(\widehat{P})-1}(\hat0,\hat1)=\varnothing,\]
    and consequently
    \[H_\ast(K_{\text{rk}(\widehat{P})}(\hat0,\hat1))\cong H_\ast(ET_{\leq\text{rk}(\widehat{P})}(\hat0,\hat1))\cong H_\ast(\Delta(P)).\]
\end{prop}
\begin{proof}
    Since $P$ is ranked, $\widehat{P}$ is also ranked. Thus, there can be no $k$-paths from $\hat0$ to $\hat1$ of length less than $\text{rk}(\widehat{P})$. Also, there can be no repeated vertices in any $k$-path of length $\text{rk}(\widehat{P})$. Thus, $K_{\text{rk}(\widehat{P})}(\hat0,\hat1)$ and $ET_{\leq\text{rk}(\widehat{P})}(\hat0,\hat1)$ are both just the simplicial complexes where the $n$-simplices are just the $n$-chains in $P$. In other words
    \[K_{\text{rk}(\widehat{P})}(\hat0,\hat1)=ET_{\leq\text{rk}(\widehat{P})}(\hat0,\hat1)=\Delta(P),\]
    which yields the isomorphism of homology groups. 
\end{proof}

\begin{thm}\label{posetTor}
    For a ranked poset $P$ and $\hat0,\hat1\in V(\mathcal{G}(\widehat{P}))$ we have
    \[MH_{\ast+2,\text{rk}(\widehat{P})}(\hat0,\hat1)\cong H_\ast(\Delta(P))\cong EMH_{\ast+2,\text{rk}(\widehat{P})}(\hat0,\hat1).\]
\end{thm}

To summarize, the existence of torsion in the order complex $\Delta(P)$, guarantees existence of  torsion in both the magnitude and Eulerian magnitude homology of the graph $G$ representing the Hasse diagram of $\widehat{P}$. 
Therefore, both  Kaneta-Yoshinaga embedding and the Asao-Izumihara isomorphism imply the existence of torsion Theorem \ref{allEMHTorsion}.


\section{Torsion without manifolds}\label{noTriang}

The Kaneta-Yoshinaga method, and in turn Asao-Izumihara,  relies on a triangulation of a manifold with torsion in its homology  in order to construct a graph whose magnitude homology theories contain torsion. 
This raises the question: Can a graph with torsion be constructed such that it is not the Hasse diagram of the face poset of a triangulable manifold? Here, we aim to answer this in the affirmative. We construct a family of posets that contain 2-torsion in their order homology and are not face posets of a triangulable manifold (with particular emphasis on ‘triangulable manifold’). However, demonstrating that these posets exhibit the desired 2-torsion requires additional insights. The first step in this process is to analyze the image of chains in the posets from Definition \ref{pksigma} under the bijection from Proposition \ref{bijection}. 

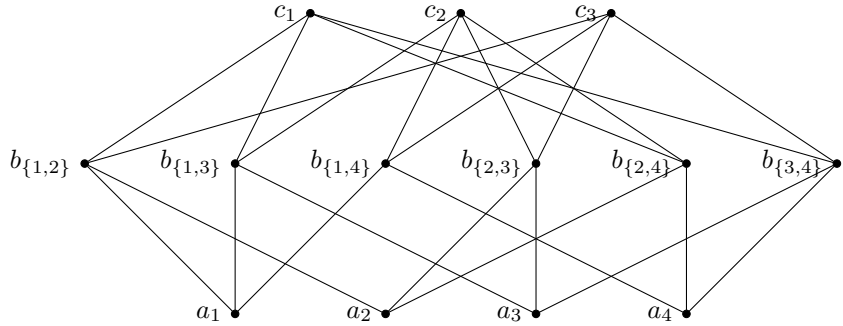
\begin{figure}[h]
    \begin{center}
        \begin{tikzpicture}
            \node[circle, draw, fill=black, scale=0.3] (a1) [label=west:$a_1$] at (-3,2) {};
            \node[circle, draw, fill=black, scale=0.3] (a2) [label=west:$a_2$] at (-1,2) {};
            \node[circle, draw, fill=black, scale=0.3] (a3) [label=west:$a_3$] at (1,2) {};
            \node[circle, draw, fill=black, scale=0.3] (a4) [label=west:$a_4$] at (3,2) {};

            \node[circle, draw, fill=black, scale=0.3] (b12) [label=west:$b_{\{1,2\}}$] at (-5,4) {};
            \node[circle, draw, fill=black, scale=0.3] (b13) [label=west:$b_{\{1,3\}}$] at (-3,4) {};
            \node[circle, draw, fill=black, scale=0.3] (b14) [label=west:$b_{\{1,4\}}$] at (-1,4) {};
            \node[circle, draw, fill=black, scale=0.3] (b23) [label=west:$b_{\{2,3\}}$] at (1,4) {};
            \node[circle, draw, fill=black, scale=0.3] (b24) [label=west:$b_{\{2,4\}}$] at (3,4) {};
            \node[circle, draw, fill=black, scale=0.3] (b34) [label=west:$b_{\{3,4\}}$] at (5,4) {};

            \node[circle, draw, fill=black, scale=0.3] (c1) [label=west:$c_1$] at (-2,6) {};
            \node[circle, draw, fill=black, scale=0.3] (c2) [label=west:$c_2$] at (0,6) {};
            \node[circle, draw, fill=black, scale=0.3] (c3) [label=west:$c_3$] at (2,6) {};
            
            \draw (a1) -- (b12) -- (a2);
            \draw (a1) -- (b13) -- (a3);
            \draw (a1) -- (b14) -- (a4);
            \draw (a2) -- (b23) -- (a3);
            \draw (a2) -- (b24) -- (a4);
            \draw (a3) -- (b34) -- (a4);
            \draw (b12) -- (c1) -- (b34);
            \draw (b13) -- (c2) -- (b24);
            \draw (b14) -- (c3) -- (b23);
            \draw (b13) -- (c1) -- (b24);
            \draw (b14) -- (c2) -- (b23);
            \draw (b12) -- (c3) -- (b34);
        \end{tikzpicture}
    \end{center}
    \caption{The poset $P_4^\sigma$}\label{f:p4}
\end{figure}

\begin{defn}\label{pksigma}
    For a positive even number $k\in2\Z_{>2}$ and a derangement $\sigma\in S_{k-1}$, let $P_k^\sigma$ be the ranked poset obtained by the following construction:
    \begin{enumerate}
        \item Let $a_i\in P$ for $i\in[k]$ such that $a_i$ covers no other element of $P$.  
        \item For $\{i,j\}\in\binom{[k]}{2}$, let $b_{\{i,j\}}$ be the element that covers $a_i$ and $a_j$.
        \item Partition the set $\binom{[k]}{2}$ into $k-1$ blocks $B_1,...,B_{k-1}$ such that every element $j\in[k]$ appears in exactly one element of each block $B_m$, $1\leq m\leq k-1$.
        \item For $1\leq m\leq k-1$, let $c_m$ cover $b_{\{i,j\}}$ if $\{i,j\}\in B_m\cup B_{\sigma(m)}$.
    \end{enumerate}
\end{defn}

By definition, every maximal chain in $P_k^\sigma$ is saturated and the rank of $P_k^\sigma$ is 2. 

\begin{example}\label{p_4^(123)}
   The Hasse diagram for $P_4^\sigma$ where $\sigma$ is the derangement $(1\;2\;3)\in S_3$ with blocks $B_1=\{\{1,2\},\{3,4\}\}$, $B_2=\{\{1,3\},\{2,4\}\}$, $B_3=\{\{1,4\},\{2,3\}\}$ is illustrated in Figure \ref{f:p4}.
\end{example}

\begin{example}\label{p_6^(21345)}
    The structure of the Hasse diagram for $P_6^\tau$ where $\tau=(2\;1\;3\;4\;5)\in S_5$, and the blocks are $B_1=\{\{1,2\},\{3,5\},\{4,6\}\}$, $B_2=\{\{1,3\},\{2,4\},\{5,6\}\}$, $B_3=\{\{1,4\},\{2,5\},\{3,6\}\}$, $B_4=\{\{1,5\},\{2,6\},\{3,4\}\}$, and $B_5=\{\{1,6\},\{2,3\},\{4,5\}\}$ is given  in Figure \ref{f:p6}.
\end{example}

\begin{figure}[h]
    \begin{center}
    \begin{tikzpicture}[scale=0.9]
        \node[circle, draw, fill=black, scale=0.3] (a1) [label=west:$a_1$] at (-5,2.5) {};
        \node[circle, draw, fill=black, scale=0.3] (a2) [label=west:$a_2$] at (-3,2.5) {};
        \node[circle, draw, fill=black, scale=0.3] (a3) [label=west:$a_3$] at (-1,2.5) {};
        \node[circle, draw, fill=black, scale=0.3] (a4) [label=west:$a_4$] at (1,2.5) {};
        \node[circle, draw, fill=black, scale=0.3] (a5) [label=west:$a_5$] at (3,2.5) {};
        \node[circle, draw, fill=black, scale=0.3] (a6) [label=west:$a_6$] at (5,2.5) {};

        \node[circle, draw, fill=black, scale=0.3] (b12) at (-7,5) {};
        \node[circle, draw, fill=black, scale=0.3] (b13) at (-6,5) {};
        \node[circle, draw, fill=black, scale=0.3] (b14) at (-5,5) {};
        \node[circle, draw, fill=black, scale=0.3] (b15) at (-4,5) {};
        \node[circle, draw, fill=black, scale=0.3] (b16) at (-3,5) {};
        \node[circle, draw, fill=black, scale=0.3] (b23) at (-2,5) {};
        \node[circle, draw, fill=black, scale=0.3] (b24) at (-1,5) {};
        \node[circle, draw, fill=black, scale=0.3] (b25) at (0,5) {};
        \node[circle, draw, fill=black, scale=0.3] (b26) at (1,5) {};
        \node[circle, draw, fill=black, scale=0.3] (b34) at (2,5) {};
        \node[circle, draw, fill=black, scale=0.3] (b35) at (3,5) {};
        \node[circle, draw, fill=black, scale=0.3] (b36) at (4,5) {};
        \node[circle, draw, fill=black, scale=0.3] (b45) at (5,5) {};
        \node[circle, draw, fill=black, scale=0.3] (b46) at (6,5) {};
        \node[circle, draw, fill=black, scale=0.3] (b56) at (7,5) {};

        \node[circle, draw, fill=black, scale=0.3] (c1) [label=west:$c_1$] at (-4,7.5) {};
        \node[circle, draw, fill=black, scale=0.3] (c2) [label=west:$c_2$] at (-2,7.5) {};
        \node[circle, draw, fill=black, scale=0.3] (c3) [label=west:$c_3$] at (0,7.5) {};
        \node[circle, draw, fill=black, scale=0.3] (c4) [label=west:$c_4$] at (2,7.5) {};
        \node[circle, draw, fill=black, scale=0.3] (c5) [label=west:$c_5$] at (4,7.5) {};

        \draw (a1) -- (b12);
        \draw (a1) -- (b13);
        \draw (a1) -- (b14);
        \draw (a1) -- (b15);
        \draw (a1) -- (b16);
        \draw (a2) -- (b12);
        \draw (a2) -- (b23);
        \draw (a2) -- (b24);
        \draw (a2) -- (b25);
        \draw (a2) -- (b26);
        \draw (a3) -- (b13);
        \draw (a3) -- (b23);
        \draw (a3) -- (b34);
        \draw (a3) -- (b35);
        \draw (a3) -- (b36);
        \draw (a4) -- (b14);
        \draw (a4) -- (b24);
        \draw (a4) -- (b34);
        \draw (a4) -- (b45);
        \draw (a4) -- (b46);
        \draw (a5) -- (b15);
        \draw (a5) -- (b25);
        \draw (a5) -- (b35);
        \draw (a5) -- (b45);
        \draw (a5) -- (b56);
        \draw (a6) -- (b16);
        \draw (a6) -- (b26);
        \draw (a6) -- (b36);
        \draw (a6) -- (b46);
        \draw (a6) -- (b56);

        \draw (b12) -- (c1);
        \draw (b35) -- (c1);
        \draw (b46) -- (c1);
        \draw (b14) -- (c1);
        \draw (b25) -- (c1);
        \draw (b36) -- (c1);
        \draw (b13) -- (c2);
        \draw (b24) -- (c2);
        \draw (b56) -- (c2);
        \draw (b12) -- (c2);
        \draw (b35) -- (c2);
        \draw (b46) -- (c2);
        \draw (b14) -- (c3);
        \draw (b25) -- (c3);
        \draw (b36) -- (c3);
        \draw (b15) -- (c3);
        \draw (b26) -- (c3);
        \draw (b34) -- (c3);
        \draw (b15) -- (c4);
        \draw (b26) -- (c4);
        \draw (b34) -- (c4);
        \draw (b16) -- (c4);
        \draw (b23) -- (c4);
        \draw (b45) -- (c4);
        \draw (b16) -- (c5);
        \draw (b23) -- (c5);
        \draw (b45) -- (c5);
        \draw (b13) -- (c5);
        \draw (b24) -- (c5);
        \draw (b56) -- (c5);

    \end{tikzpicture}
    \end{center}
    \caption{The poset $P_6^\tau$}\label{f:p6}
\end{figure}
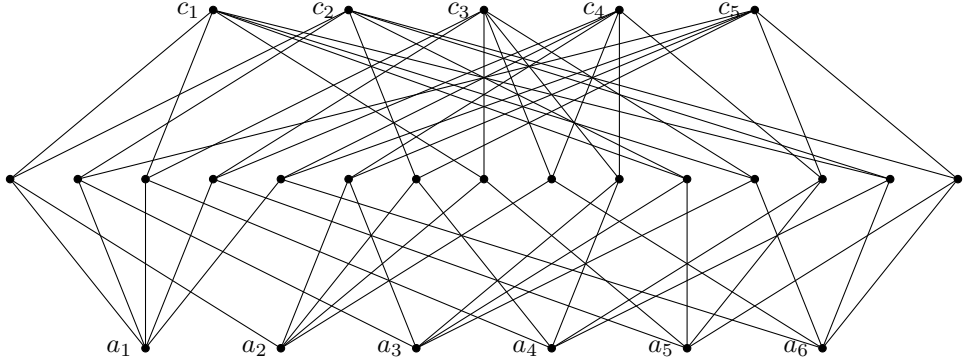

\begin{prop}
    For $k\in\Z_{>2}$, $P_k^\sigma$ is not the face poset of the triangulation of a manifold.
\end{prop}
\begin{proof}
    If it were, then $P_k^\sigma$ would have to be the triangulation of a 2-manifold since it is a ranked poset of rank 2. Yet every element which would represent a facet, by construction, will cover more than three other elements, which yields a contradiction. Therefore, $P_k^\sigma$ cannot be the face poset of a triangulable manifold.
\end{proof}

\begin{defn}
    Given $k\in2\Z_{>2}$, a $k$-gon is assigned an \textbf{alternating coloring} if each side of the $k$-gon alternates in color. If $K$ is a $k$-gon which is assigned an alternating coloring, then it is said that $K$ is \textbf{alternately colored}.
\end{defn}

\begin{prop}
    Given $n\in\Z_{>3}$, let $AC_k(n)$ be the set of $k$-gons such that there are $n$ colors, each $k$-gon is alternately colored, and each color appears in exactly two $k$-gons. Then 
    \begin{equation}
        |AC_k(n)|=n.
    \label{eq1}
    \end{equation}
\end{prop}
\begin{proof}
    First we shall prove that $AC_k(n)$ has the initial condition $|AC_k(3)|=3.$ 
    For $n=3$, choose two colors $\ob{A}$ and $\ob{B}$ for the first $k$-gon. Next place $\ob{A}$ in one $k$-gon and $\ob{B}$ in another $k$-gon. Now choose another color $\ob{C}$ to place in each $k$-gon. Thus there are three $k$-gons: one colored with $\ob{A}$ and $\ob{B}$, one colored with $\ob{A}$ and $\ob{C}$, and one colored with $\ob{B}$ and $\ob{C}$.
    Now for the set $AC_k(n)$, take any one of the $k$-gons $K\in AC_k(n)$ with colors $\ob{C_1}$ and $\ob{C_2}$ and ``split" it into two new $k$-gons, $K_1$ and $K_2$, such that $K_1$ is alternately colored with $\ob{C_1}$ and $K_2$ is alternately colored with $\ob{C_2}$. Then, color the remaining sides of $K_1$ and $K_2$ with a new color $\ob{C_3}$. Since the colors $\ob{C_1}$ and $\ob{C_2}$ already appear in another $k$-gon, and by construction $\ob{C_3}$ appears in the two $k$-gons $K_1$ and $K_2$, we have
    \[AC_k(n+1)=AC_k(n)\setminus\{K\}\uplus\{K_1,K_2\}.\]
    Therefore, 
    \begin{eqnarray*}
        |AC_k(n+1)| &=& |AC_k(n)\setminus\{K\}|+|\{K_1,K_2\}| \\
                  &=& |AC_k(n)|-1+2 \\
                  &=& |AC_k(n)|+1,
    \end{eqnarray*} 
    and thus $|AC_k(n)|=n$.
\end{proof}

Recall that for a subset $S$ of a poset $P$, the lower order ideal $I(S)$ of $S$ is the subposet of all elements $p\in P$ such that $p\leq s$ for all $s\in S$. For any singleton subset $\{s\}\subseteq P$ we write $I(s)$ instead of $I(\{s\})$.

\begin{defn}
    Let $\I_k^\sigma$ be set of all lower order ideals of the form $I(c_m)$ for $c_m\in P_k^\sigma$.
\end{defn}

\begin{prop}\label{bijection}
    There exists a bijection $\varphi:\mathcal{I}_k^\sigma\to AC_k(k-1)$.
\end{prop}
\begin{proof}
    Assign each block $B_i$ in the construction of $P_k^\sigma$ a distinct color $\ob{B_i}$. Then $\varphi(I(c_m))$ is the $k$-gon with alternating colors $\ob{B_m}$ and $\ob{B_{\sigma(m)}}$. The map is well-defined since each color $\ob{B_m}$ will appear in exactly two $k$-gons: $\varphi(I(c_m))$ and $\varphi(I(c_{\sigma^{-1}(m)}))$. 
\end{proof}
    
Note that, because of the way $P_k^\sigma$ was constructed, the $k/2$ elements $b_{\{i,j\}}$ for $\{i,j\}\in B_m$ are covered by exactly two elements: $c_m$ and $c_{\sigma^{-1}(m)}$. This corresponds to the fact that each $\varphi(I(c_m))$ must contain two colors and that each color must also be in exactly one other $k$-gon. Hence, the elements $b_{\{i,j\}}\in I(c_m)$ correspond to the colored edges of the $k$-gons since there are precisely $k$ elements $b_{\{i,j\}}$ in each lower order ideal $I(c_m)$. Again, because of the way $P_k^\sigma$ was constructed, each $i\in[k]$ is in exactly two sets corresponding to the two blocks $B_m$ and $B_{\sigma(m)}$ so that there are exactly two other $j,f\in[k]$ which are ``adjacent" to $i$. What this means, is that the vertices of $\varphi(I(c_m))$ are the elements $a_i\in I(c_m)$ arranged in such a way that $a_i$ is adjacent to both $a_j$ and $a_f$ in $\varphi(I(c_m))$ if $b_{\{i,j\}},b_{\{i,f\}}\in I(c_m)$. Therefore, each $a_i$ represents the vertices, each $b_{\{i,j\}}$ represents the colored edges, and $c_m$ represents the face of the alternately colored $k$-gon $\varphi(I(c_m))$.

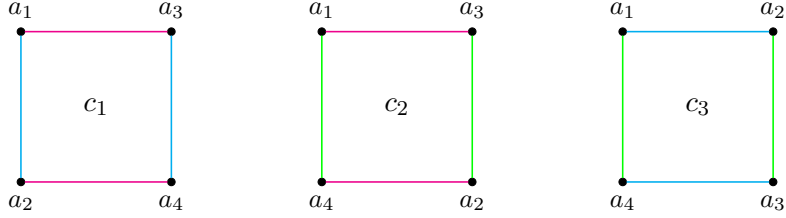
\begin{figure}[h]
    \begin{center}
        \begin{tikzpicture}
            \node[circle, draw, fill=black, scale=0.3] (0) [label=north:$a_1$] at (0,0) {};
            \node[circle, draw, fill=black, scale=0.3] (1) [label=south:$a_2$] at (0,-2) {};
            \node[circle, draw, fill=black, scale=0.3] (2) [label=south:$a_4$] at (2,-2) {};
            \node[circle, draw, fill=black, scale=0.3] (3) [label=north:$a_3$] at (2,0) {};

            \node[circle, draw, fill=black, scale=0.3] (4) [label=north:$a_1$] at (4,0) {};
            \node[circle, draw, fill=black, scale=0.3] (5) [label=south:$a_4$] at (4,-2) {};
            \node[circle, draw, fill=black, scale=0.3] (6) [label=south:$a_2$] at (6,-2) {};
            \node[circle, draw, fill=black, scale=0.3] (7) [label=north:$a_3$] at (6,0) {};

            \node[circle, draw, fill=black, scale=0.3] (8) [label=north:$a_1$] at (8,0) {};
            \node[circle, draw, fill=black, scale=0.3] (9) [label=south:$a_4$] at (8,-2) {};
            \node[circle, draw, fill=black, scale=0.3] (10) [label=south:$a_3$] at (10,-2) {};
            \node[circle, draw, fill=black, scale=0.3] (11) [label=north:$a_2$] at (10,0) {};

            \node at (1,-1) {\large $c_1$};
            \node at (5,-1) {\large $c_2$};
            \node at (9,-1) {\large $c_3$};

            \draw[cyan, semithick] (0) -- (1);
            \draw[magenta, semithick] (1) -- (2);
            \draw[cyan, semithick] (2) -- (3);
            \draw[magenta, semithick] (3) -- (0);

            \draw[green, semithick] (4) -- (5);
            \draw[magenta, semithick] (5) -- (6);
            \draw[green, semithick] (6) -- (7);
            \draw[magenta, semithick] (7) -- (4);

            \draw[green, semithick] (8) -- (9);
            \draw[cyan, semithick] (9) -- (10);
            \draw[green, semithick] (10) -- (11);
            \draw[cyan, semithick] (11) -- (8);
        \end{tikzpicture}
    \end{center}
    \caption{The bijection $\varphi:\mathcal{I}_4^\sigma\to AC_4(3)$}
    \label{f:I4}
\end{figure}

\begin{example} 
    For the poset $P_4^\sigma$ in Figure \ref{f:p4}, the bijection $\varphi:\mathcal{I}_4^\sigma\to AC_4(3)$ yields the alternately colored squares in Figure \ref{f:I4}.
\end{example}

\begin{example} The bijection $\varphi:\I_6^\tau\to AC_6(5)$ 
    for the poset $P_6^\tau$ in Figure \ref{f:p6} gives alternately colored hexagons shown in Figure \ref{f:I6}.
\end{example}

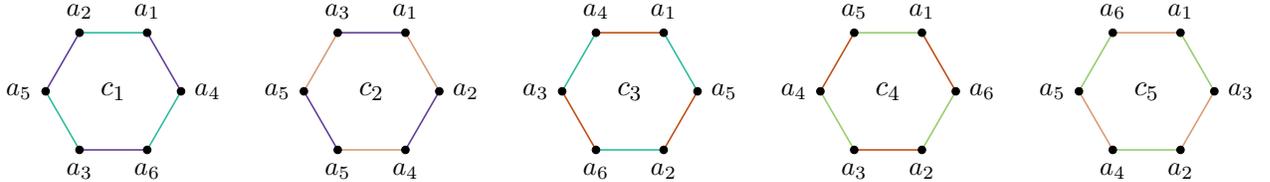
\begin{figure}
    \centering
    \begin{tabular}{ccccc}
    \begin{tikzpicture}[scale=0.45]
        \newdimen\R
        \R=2cm
        \draw[RoyalPurple, semithick] (0:\R) \foreach \x in {120,240,360} { (\x:\R) -- (\x+60:\R)};
        \draw[SeaGreen, semithick] (0:\R) \foreach \x in {60,180,300}{(\x:\R) -- (\x+60:\R)};
        \foreach \x/\l/\p in
            { 60/{$a_1$}/above,
            120/{$a_2$}/above,
            180/{$a_5$}/left,
            240/{$a_3$}/below,
            300/{$a_6$}/below,
            360/{$a_4$}/right
        }
        \node[inner sep=1pt,circle,draw,fill,label={\p:\l}] at (\x:\R) {};
        \node at (0,0) {\large $c_1$};
    \end{tikzpicture} &

    \begin{tikzpicture}[scale=0.45]
        \newdimen\R
        \R=2cm
        \draw[Tan, semithick] (0:\R) \foreach \x in {120,240,360} { (\x:\R) -- (\x+60:\R)};
        \draw[RoyalPurple, semithick] (0:\R) \foreach \x in {60,180,300}{(\x:\R) -- (\x+60:\R)};
        \foreach \x/\l/\p in
            { 60/{$a_1$}/above,
            120/{$a_3$}/above,
            180/{$a_5$}/left,
            240/{$a_5$}/below,
            300/{$a_4$}/below,
            360/{$a_2$}/right
        }
        \node[inner sep=1pt,circle,draw,fill,label={\p:\l}] at (\x:\R) {}; 
        \node at (0,0) {\large $c_2$};
        \end{tikzpicture} &

        \begin{tikzpicture}[scale=0.45]
        \newdimen\R
        \R=2cm
        \draw[SeaGreen, semithick] (0:\R) \foreach \x in {120,240,360} { (\x:\R) -- (\x+60:\R)};
        \draw[Bittersweet, semithick] (0:\R) \foreach \x in {60,180,300}{(\x:\R) -- (\x+60:\R)};
        \foreach \x/\l/\p in
            { 60/{$a_1$}/above,
            120/{$a_4$}/above,
            180/{$a_3$}/left,
            240/{$a_6$}/below,
            300/{$a_2$}/below,
            360/{$a_5$}/right
        }
        \node[inner sep=1pt,circle,draw,fill,label={\p:\l}] at (\x:\R) {};
        \node at (0,0) {\large $c_3$};
    \end{tikzpicture} 
    &
    \begin{tikzpicture}[scale=0.45]
        \newdimen\R
        \R=2cm
        \draw[Bittersweet, semithick] (0:\R) \foreach \x in {120,240,360} { (\x:\R) -- (\x+60:\R)};
        \draw[YellowGreen, semithick] (0:\R) \foreach \x in {60,180,300}{(\x:\R) -- (\x+60:\R)};
        \foreach \x/\l/\p in
            { 60/{$a_1$}/above,
            120/{$a_5$}/above,
            180/{$a_4$}/left,
            240/{$a_3$}/below,
            300/{$a_2$}/below,
            360/{$a_6$}/right
        }
        \node[inner sep=1pt,circle,draw,fill,label={\p:\l}] at (\x:\R) {};
        \node at (0,0) {\large $c_4$};
    \end{tikzpicture} & 

    \begin{tikzpicture}[scale=0.45]
        \newdimen\R
        \R=2cm
        \draw[YellowGreen, semithick] (0:\R) \foreach \x in {120,240,360} { (\x:\R) -- (\x+60:\R)};
        \draw[Tan, semithick] (0:\R) \foreach \x in {60,180,300}{(\x:\R) -- (\x+60:\R)};
        \foreach \x/\l/\p in
            { 60/{$a_1$}/above,
            120/{$a_6$}/above,
            180/{$a_5$}/left,
            240/{$a_4$}/below,
            300/{$a_2$}/below,
            360/{$a_3$}/right
        }
        \node[inner sep=1pt,circle,draw,fill,label={\p:\l}] at (\x:\R) {};
        \node at (0,0) {\large $c_5$};
    \end{tikzpicture}
    \end{tabular}
    \caption{The bijection $\varphi:\I_6^\tau\to AC_6(5)$}
    \label{f:I6}
\end{figure}

Even though $A_k(k-1)$ is considered to be a set of alternately colored $k$-gons with unlabeled vertices, the element $\varphi(I(c_m))$ will always be considered to have labeled vertices and faces. From here on, $\varphi$ will always refer to the bijection in Proposition \ref{bijection}.

\begin{prop}
    Let $I(c_m)\in\I_k^\sigma$. Then each (maximal) 2-chain in $I(c_m)$ corresponds to a directed edge in $\varphi(I(c_m))$.
\end{prop}
\begin{proof}
    Consider the chain $a_i<b_{\{i,j\}}<c_m$ in $I(c_m)$. Assume that this chain corresponds to the edge that is directed away from the vertex $a_i$ in the $k$-gon $\varphi(I(c_m))$. Then each chain corresponds to a unique directed edge in each $k$-gon $\varphi(I(c_m))$ since each edge $b_{\{i,j\}}$ only appears once in exactly two $k$-gons.
\end{proof}

Since the maximal chains in $\mathcal{I}_k^\sigma$ are maximal chains in $P_k^\sigma$, the following becomes an immediate consequence.

\begin{cor}
    Each maximal chain $a_i<b_{\{i,j\}}<c_m$ in $P_k^\sigma$ corresponds to a unique directed edge in $\varphi(I(c_m))$.
\end{cor}

The last correspondence which will be important to this discussion, is the fact that the bijection $\varphi$ guarantees a unique ordering of the vertices of each $k$-gon up to rotational and reflectional symmetry.

\begin{prop}
    Let $I(c_m),I(c_{\sigma(m)})\in\I_k^\sigma$. Then for the alternately colroed $k$-gons $\varphi(I(c_m))$ and $\varphi(I(c_{\sigma(m)}))$ which share a color, no matter the ordering of the vertices around each $k$-gon at least one pair of adjacent vertices will be ordered in reverse in each $k$-gon.
\end{prop}
\begin{proof}
    This is because if every pair of such vertices $a_{i_1}$ and $a_{i_2}$ for $\{i_1,i_2\}\in B_m$ were ordered the same way in $\varphi(I(c_{\sigma(m)}))$ that they were in $\varphi(I(c_m))$, the fact that $\varphi$ is a bijection implies that $I(c_m)=I(c_{\sigma(m)})$ and furthermore that $\sigma(m)=m$ making $m$ a fixed point of $\sigma$, a contradiction. 
\end{proof}

\begin{cor}\label{reverse_kgon}
    Let $I(c_m),I(c_{\sigma^{-1}(m)})\in\I_k^\sigma$. Then for any two alternately colroed $k$-gons $\varphi(I(c_m))$ and $\varphi(I(c_{\sigma^{-1}(m)}))$ which share a color, no matter the ordering of the vertices around the $k$-gon at least one pair of adjacent vertices will be ordered in reverse in each $k$-gon.  
\end{cor}

\begin{prop}\label{bdry_prop}
    Any 1-chain of the form $(x_0<x_1)$ in the order complex $\Delta(P_k^\sigma)$ appears (having the same sign) in the boundary of exactly two 2-chains. 
\end{prop}
\begin{proof}
The proof involves analyzing three cases:\begin{enumerate}
    \item[A)]  Consider the 1-chain $(b_{\{i,j\}}<c_m)$. Since each element $b_{\{i,j\}}\in P_k^\sigma$ covers exactly two elements $a_i,a_j\in P_k^\sigma$, there are only two vertices which can be added to $(b_{\{i,j\}}<c_m)$ to make it a 2-chain in $P_k^\sigma$: the vertices $a_i$ and $a_j$. Hence, the only 2-chains of $P_k^\sigma$ which contain $-(b_{\{i,j\}}<c_m)$ in their boundaries are $(a_i<b_{\{i,j\}}<c_m)$ and $(a_j<b_{\{i,j\}}<c_m)$.

    \item[B)]  Consider the 1-chain $(a_i<c_m)$. By the way that $P_k^\sigma$ was constructed, particularly in step 3, the element $i\in[k]$ appears exactly once in each block $B_m$ and $B_{\sigma(m)}$. Thus for $\{i,j\}\in B_m$ and $\{i,h\}\in B_{\sigma(m)}$ the only two vertices which can be inserted into $(a_i<c_m)$ to make a 2-chain in $P_k^\sigma$ are $b_{\{i,j\}}$ and $b_{\{i,h\}}$. Hence, the only 2-chains which have $(a_i<c_m)$ in their boundary are $(a_i<b_{\{i,j\}}<c_m)$ and $(a_i<b_{\{i,h\}}<c_m)$.

     \item[C)] The last case is that of  the 1-chain $(a_i<b_{\{i,j\}})$. Since $B_1,...,B_{k-1}$ is a partition of $\binom{[k]}{2}$, the unordered pair $\{i,j\}$ appears in only one of those blocks. Let $B_m$ be this block. Because of step 4 of the construction of $P_k^\sigma$, there are only two elements of $P_k^\sigma$ which cover $b_{\{i,j\}}$: $c_m$ and $c_{\sigma^{-1}(m)}$. Hence, the only two 2-chains with $-(a_i<b_{\{i,j\}})$ in their boundary are $(a_i<b_{\{i,j\}}<c_m)$ and $(a_i<b_{\{i,j\}}<c_{\sigma^{-1}(m)})$. 
\end{enumerate}
       
\end{proof}

For any finite ranked poset $P$ with rank $q$, the chain group $C_n(\Delta(P))$ is generated by the maximal chains of $P$. 
Given the correspondence between maximal chains of 
$P_k^\sigma$
  and alternately colored 
$k$-gons, how does the boundary operator interact with this correspondence?

\begin{prop}\label{Prop32}
    Consider the bijection $\varphi:\I_k^\sigma\to AC_k(k-1)$. Two 2-chains in the order complex $\Delta(P_k^\sigma)$ share a 1-chain in their boundary if one of the three conditions are met:
    \begin{itemize}
        \item[C1.] For $I(c_m)\in\I_k^\sigma$, they share the same edge in $\varphi(I(c_m))$ but have opposing directions.
        \item[C2.] For $I(c_m)\in\I_k^\sigma$, they are both directed away from the same vertex in $\varphi(I(c_m))$.
        \item[C3.] For $I(c_m),I(c_{\sigma^{-1}(m)})\in\I_k^\sigma$, their directed edges are the same color in each $k$-gon $\varphi(I(c_m))$ and $\varphi(I(c_{\sigma^{-1}(m)}))$ but they are directed away from the same vertex.
    \end{itemize}
\end{prop}
\begin{proof}
    For C1, if the two 2-chains share the same edge in $\varphi(I(c_m))$ with vertices $a_i$ and $a_j$, then they will be of the form $(a_i<b_{\{i,j\}}<c_m)$ and $(a_j<b_{\{i,j\}}<c_m)$. Thus they will share the 1-chain $(b_{\{i,j\}}<c_m)$ in their boundaries. For C2, if the two 2-chains are directed away from the vertex $a_i$ in $\varphi(I(c_m))$ with edges $b_{\{i,j\}}$ and $b_{\{i,h\}}$, then they will be the chains $(a_i<b_{\{i,j\}}<c_m)$ and $(a_i<b_{\{i,h\}}<c_m)$ and they will share the 1-chain $(a_i<c_m)$ in their boundaries. Note that for C3, if the two 2-chains correspond to edges in different $k$-gons with the same color but directed away from the same vertex $a_i$, then their colored edges must be represented by the same $b_{\{i,j\}}$ because of the way $P_k^\sigma$ and $\varphi$ were constructed. Hence, any two 2-chains like those in C3 must be of the form $(a_i<b_{\{i,j\}}<c_m)$ and $(a_i<b_{\{i,j\}}<c_{\sigma^{-1}(m)})$. Thus, they share the same 1-chain $(a_i<b_{\{i,j\}})$ in their boundaries.
\end{proof}

\begin{example}\label{Ex}
    For $\Delta(P_4^\sigma)$ with $P_4^\sigma$ as in Example \ref{p_4^(123)}, the two 2-chains $(a_1<b_{\{1,3\}}<c_1)$ and $(a_3<b_{\{1,3\}}<c_1)$ are represented by the directed edges in Figure \ref{f:p4c1}, and both have $-(b_{\{1,3\}}<c_1)$ in their boundaries. The two 2-chains $(a_1<b_{\{1,3\}}<c_1)$ and $(a_1<b_{\{1,2\}}<c_1)$ are represented by the directed edges in Figure \ref{f:p4c2}, and both have $(a_1<c_1)$ in their boundaries. Last, the two 2-chains $(a_1<b_{\{1,3\}}<c_1)$ and $(a_1<b_{\{1,3\}}<c_2)$ are represented by the directed edges in Figure \ref{f:p4c3}, and both have $-(a_1<b_{\{1,3\}})$ in their boundaries.
\end{example}

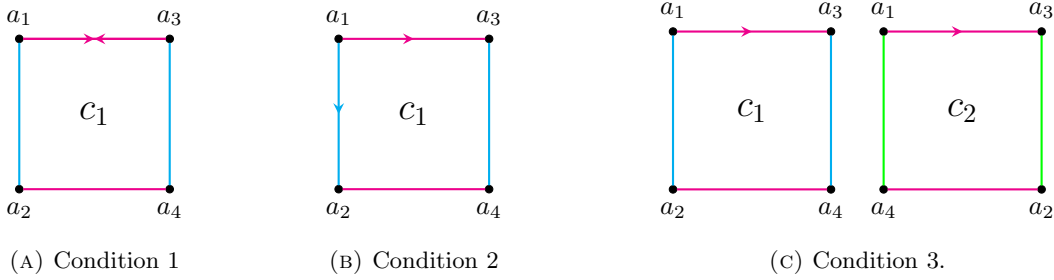
\begin{figure}[h]

\begin{subfigure}{0.25\textwidth}
    \begin{center}
    \begin{tikzpicture}[scale=0.50]
        \node[circle, draw, fill=black, scale=0.3] (0) [label=north:$a_1$] at (0,0) {};
        \node[circle, draw, fill=black, scale=0.3] (1) [label=south:$a_2$] at (0,-4) {};
        \node[circle, draw, fill=black, scale=0.3] (2) [label=south:$a_4$] at (4,-4) {};
        \node[circle, draw, fill=black, scale=0.3] (3) [label=north:$a_3$] at (4,0) {};
        \node at (2,-2) {\LARGE $c_1$};

        \draw[cyan, thick] (0) -- (1);
        \draw[magenta, thick] (1) -- (2);
        \draw[cyan, thick] (2) -- (3);
        \draw[magenta, thick, postaction={on each segment={mid arrow=magenta}}] (3) -- (0);
        \draw[magenta, thick, postaction={on each segment={mid arrow=magenta}}] (0) -- (3);

    \end{tikzpicture}
    \end{center}
    \caption{Condition 1}\label{f:p4c1}
    \end{subfigure}
    \begin{subfigure}{0.25\textwidth}
    \begin{center} 
    \begin{tikzpicture}[scale=0.5]
        \node[circle, draw, fill=black, scale=0.3] (0) [label=north:$a_1$] at (0,0) {};
        \node[circle, draw, fill=black, scale=0.3] (1) [label=south:$a_2$] at (0,-4) {};
        \node[circle, draw, fill=black, scale=0.3] (2) [label=south:$a_4$] at (4,-4) {};
        \node[circle, draw, fill=black, scale=0.3] (3) [label=north:$a_3$] at (4,0) {};
        \node at (2,-2) {\LARGE $c_1$};

        \draw[cyan, thick, postaction={on each segment={mid arrow=cyan}}] (0) -- (1);
        \draw[magenta, thick] (1) -- (2);
        \draw[cyan, thick] (2) -- (3);
        \draw[magenta, thick, postaction={on each segment={mid arrow=magenta}}] (0) -- (3);
    \end{tikzpicture}
    \end{center}
    \caption{Condition 2 }\label{f:p4c2}
    \end{subfigure}
    \begin{subfigure}{0.45\textwidth}
    \begin{center}     
    \begin{tikzpicture}[scale=0.7]
        \node[circle, draw, fill=black, scale=0.3] (0) [label=north:$a_1$] at (0,0) {};
        \node[circle, draw, fill=black, scale=0.3] (1) [label=south:$a_2$] at (0,-3) {};
        \node[circle, draw, fill=black, scale=0.3] (2) [label=south:$a_4$] at (3,-3) {};
        \node[circle, draw, fill=black, scale=0.3] (3) [label=north:$a_3$] at (3,0) {};
        \node at (1.5,-1.5) {\LARGE $c_1$};

        \node[circle, draw, fill=black, scale=0.3] (4) [label=north:$a_1$] at (4,0) {};
        \node[circle, draw, fill=black, scale=0.3] (5) [label=south:$a_4$] at (4,-3) {};
        \node[circle, draw, fill=black, scale=0.3] (6) [label=south:$a_2$] at (7,-3) {};
        \node[circle, draw, fill=black, scale=0.3] (7) [label=north:$a_3$] at (7,0) {};
        \node at (5.5,-1.5) {\LARGE $c_2$};

        \draw[cyan, thick] (0) -- (1);
        \draw[magenta, thick] (1) -- (2);
        \draw[cyan, thick] (2) -- (3);
        \draw[magenta, thick, postaction={on each segment={mid arrow=magenta}}] (0) -- (3);

        \draw[green, thick] (4) -- (5);
        \draw[magenta, thick] (5) -- (6);
        \draw[green, thick] (6) -- (7);
        \draw[magenta, thick, postaction={on each segment={mid arrow=magenta}}] (4) -- (7);
    \end{tikzpicture}
    \end{center}
    \caption{Condition 3.}\label{f:p4c3}
\end{subfigure}
\caption{Conditions on  $\Delta(P_4^\sigma)$   in Example \ref{Ex}.}
\end{figure}

Next, consider the multi-index notation $J=(i,\{i,j\},m)$ for $i\in[k]$, $\{i,j\}\in B_m\cup B_{\sigma(m)}$, and $m\in[k-1]$. For any chain 
\begin{equation}
    \beta=\sum_{\substack{i\in[k] \\ \{i,j\}\in B_m\cup B_{\sigma(m)}\\ m\in[k-1]}}v_J(a_i<b_{\{i,j\}}<c_m)\in C_2(\Delta(P_k^\sigma))
    \label{beta}
\end{equation}
with $v_J\in\Z$, we want to use the correspondence between maximal chains and alternately colored $k$-gons to show that $\partial\beta\neq\alpha$ where
\[\alpha=\sum_{\substack{\{i,j\}\in B_m\cup B_{\sigma(m)} \\ m\in[k-1]}}(b_{\{i,j\}}<c_m)\;-\sum_{\substack{i\in[k] \\ m\in[k-1]}}(a_i<c_m)\;+\sum_{\substack{i,j\in[k] \\ i\neq j}}(a_i<b_{\{i,j\}}).\]
In the discussion to follow, $\alpha$ will always refer to the above 1-chain.

\begin{prop}
    Let $\beta$ be as in Equation \ref{beta}. Then in order for $\partial\beta$ to have a coefficient of $1$ next to $(b_{\{i,j\}}<c_m)$, the coefficients $v_{J_1},v_{J_2}\in\Z$ must satisfy
    \begin{equation}
        v_{J_1}+v_{J_2}=1
    \label{coeffs1}
     \end{equation}
    for $J_1=(i,\{i,j\},m)$ and $J_2=(j,\{i,j\},m)$.
\end{prop}
\begin{proof}
    Proposition \ref{bdry_prop} insures that the only two 2-chains which contain $-(b_{\{i,j\}}<c_m)$ in their boundary are $(a_i<b_{\{i,j\}}<c_m)$ and $(a_j<b_{\{i,j\}}<c_m)$. Now $v_{J_1}$ is the coefficient for $(a_i<b_{\{i,j\}}<c_m)$ and the coefficient for $(a_j<b_{\{i,j\}}<c_m)$ is $v_{J_2}$. Thus, in order to get a coefficient of $-1$ next to $(b_{\{i,j\}}<c_m)$ in $\partial\beta$, Equation \ref{coeffs1} must be satisfied.
\end{proof}

Similar proofs also show that similar equations must be satisfied for the 1-chains $(a_i<b_{\{i,j\}})$ and $(a_i<c_m)$ to have coefficients of $1$ and $-1$, respectively, in the boundary of a 2-chain.

\begin{cor}
    Let $\beta$ be as in Equation \ref{beta}. Then in order for $\partial\beta$ to have a coefficient of $-1$ next to $(a_i<c_m)$, the coefficients $v_{H_1},v_{H_2}\in\Z$ must satisfy
    \begin{equation}
        v_{H_1}+v_{H_2}=1
    \label{coeffs2}
    \end{equation}
    for $H_1=(i,\{i,j\},m)$, $H_2=(i,\{i,h\},m)$, and $\{i,j\},\{i,h\}\in B_m$.
\end{cor}

\begin{cor}
    Let $\beta$ be as in Equation \ref{beta}. Then in order for $\partial\beta$ to have a coefficient of $1$ next to $(a_i<b_{\{i,j\}})$, the coefficients $v_{L_1},v_{L_2}\in\Z$ must satisfy
    \begin{equation}
        v_{L_1}+v_{L_2}=1
    \label{coeffs3}
    \end{equation}
    for $L_1=(i,\{i,j\},m)$ and $L_2=(i,\{i,j\},\sigma^{-1}(m))$.
\end{cor}

\begin{prop}\label{notaboundary}
    Let $\beta$ be as in Equation \ref{beta}. Then $\partial\beta\neq\alpha$.
\end{prop}
\begin{proof}
    If $\partial\beta=\alpha$ then Equation \ref{coeffs1}, Equation \ref{coeffs2}, and Equation \ref{coeffs3} must all be satisfied. Let $J=(i,\{i,j\},m)$ and let $v_J=z\in\Z$. Notice that  $v_J$ corresponds to the coefficient next to the 2-chain $(a_i<b_{\{i,j\}}<c_m)$, and this corresponds to a maximal chain of $P_k^\sigma$. Next maximal chains of $P_k^\sigma$ can be represented as directed edges through the bijection $\varphi$. Hence, $(a_i<b_{\{i,j\}}<c_m)$ represents a directed edge in $\varphi(I(c_m))$. Thus, assigning coefficients to 2-chains becomes the same as assigning weights to the directed edges of each $k$-gon $\varphi(I(c_m))$. Let the directed edge representing $(a_i<b_{\{i,j\}}<c_m)$ be directed in the clockwise direction around the $k$-gon $\varphi(I(c_m))$. Then this edge has the weight $z$.
    Condition C1 from Proposition \ref{Prop32} and Equation \ref{coeffs1} imply that the same edge having the opposite direction will have a weight of $1-z$. Proposition \ref{Prop32} Case C2 and Equation \ref{coeffs2} imply that the other edge sharing the same vertex will have a weight of $1-z$. The last case C3  in Proposition \ref{Prop32} and Equation \ref{coeffs3} imply that the same edge directed in the same way (directed away from $a_i$, but not necessarily clockwise) in the $k$-gon $\varphi(I(c_{\sigma^{-1}(m)}))$ will have a weight of $1-z$. Then repeatedly applying C1 with Equation \ref{coeffs1} and C2 with Equation \ref{coeffs2} to $\varphi(I(c_m))$ yields that all the edges directed clockwise must have a weight of $z$ while all the edges directed counterclockwise must have a weight of $1-z$. Since the ordering of the vertices of each $k$-gon are unique, arrange $\varphi(I(c_{\sigma^{-1}(m)}))$ such that the edge representing $(a_i<b_{\{i,j\}}<c_m)$ is directed clockwise. Corollary \ref{reverse_kgon} guarantees that at least one of the edges which $\varphi(I(c_m))$ and $\varphi(I(c_{\sigma^{-1}(m)}))$ share is ordered in reverse in each $k$-gon. Applying C1 with Equation \ref{coeffs1} and C2 with Equation \ref{coeffs2} to $\varphi(I(c_{\sigma^{-1}(m)}))$ yields that every edge directed clockwise in $\varphi(I(c_{\sigma^{-1}(m)}))$ has a weight of $1-z$, and thus at least one of the directed edges $(a_{i^\ast}<b_{\{i^\ast,j^\ast\}}<c_{\sigma^{-1}(m)})$ for $\{i^\ast,j^\ast\}\in B_m$ has a weight of $1-z$. However, $(a_{i^\ast}<b_{\{i^\ast,j^\ast\}}<c_m)$ will also have a weight of $1-z$ in $\varphi(I(c_m))$. Since the weights for these edges correspond to the coefficients $v_{L_1}$ and $v_{L_2}$, where $L_1=(i^\ast,\{i^\ast,j^\ast\},c_m)$ and $L_2=(i^\ast,\{i^\ast,j^\ast\},\sigma^{-1}(m))$, we have
    \[v_{L_1}+v_{L_2}=2-2z\neq 1\]
    for any $z\in\Z$. This contradicts Equation \ref{coeffs3}, implying that Equations \ref{coeffs1}, \ref{coeffs2}, and \ref{coeffs3} cannot all be satisfied simultaneously. 
\end{proof}

\begin{thm}
    There exists a subgroup $\Z_2\trianglelefteq H_1(\Delta(P_k^\sigma))$ generated by $[\alpha]$.
\end{thm}
\begin{proof}
   Proposition \ref{notaboundary} proves that $\alpha$ is not a boundary so that $\alpha\notin\im{\partial_2}$. Now consider the element 
    \[\gamma=\sum_{\substack{i\in[k]\\ \{i,j\}\in B_m\cup B_{\sigma(m)}\\ m\in[k-1]}}(a_i<b_{\{i,j\}}<c_m)\in C_2(\Delta(P_k^\sigma)).\]
    Proposition \ref{bdry_prop} then implies that $\partial\gamma=2\alpha$ so that $2\alpha\in\im{\partial_2}$ and consequently $\alpha\in\ker{\partial_1}$. Therefore 
    \[\Z_2\cong\langle[\alpha]\rangle\trianglelefteq H_1(\Delta(P_k^\sigma)).\]
\end{proof}

Let $G_k^\sigma$ be the graph with vertex set $\widehat{P_k^\sigma}$ and edge set $\{\{a,b\}\,|\,a\lessdot b\in \widehat{P_k^\sigma}\}$. The previous theorem can then be restated in terms of magnitude and Eulerian magnitude homology via Theorem \ref{posetTor}. 

\begin{cor}
    There exists a subgroup $\Z_2\trianglelefteq MH_{3,4}(G_k^\sigma)$ and a subgroup $\Z_2\trianglelefteq EMH_{3,4}(G_k^\sigma)$ both generated by $[\alpha]$ where
    \[\alpha=-\sum_{\substack{\{i,j\}\in B_m\cup B_{\sigma(m)} \\ m\in[k-1]}}(\hat 0,b_{\{i,j\}},c_m,\hat 1)\;+\sum_{\substack{i\in[k] \\ m\in[k-1]}}(\hat 0,a_i,c_m,\hat 1)\;-\sum_{\substack{i,j\in[k] \\ i\neq j}}(\hat 0,a_i,b_{\{i,j\}},\hat 1).\]
\end{cor}
 \section{Regular CW structures and torsion}\label{CW}
Upon closer examination of the posets $P_k^\sigma$ and the isomorphism $\varphi$, it becomes evident that the $k$-gons can be realized as a CW structure by appropriately identifying edges and vertices. As it turns out, the elements of $P_k^\sigma$ are all cells of a CW structure on $\mathbb{RP}^2$ ordered by topological closure inclusion. Thus, the posets $P_k^\sigma$ are not face posets of a trignaulation, but rather are face posets of a CW complex. Thus far the only examples of torsion in magnitude homology come from posets with 2-torsion. 

Our goal is to provide additional examples of posets (and consequently graphs) with 
p-torsion in their order (and magnitude) homology groups for $p>2$.
 We also aim to present a poset with a rank higher than 2 that contains torsion in its order homology.  
Since triangulating is computationally intensive, our goal was to find a more efficient way of producing the desired posets and graphs. Kaneta and Yoshinaga   \cite{Kaneta_Yoshinaga} use results from  \cite{kozlov2007combinatorial} when justifying the fact that the barycentric subdivision $\text{Bd}(M)$ of the triangulation of a manifold is homeomorphic to the geometric realization $|\Delta(F(M))|$ of the order complex of the face poset of $M$. In this section we use a stronger result from Kozlov's book. 

\begin{defn}
    For a regular CW complex $K$, let $\mathcal{F}(K)$ denote the the poset of all closures of nonempty cells of $K$ ordered by inclusion.
\end{defn}

\begin{prop}[\cite{kozlov2007combinatorial}]\label{Koz}
    Any arbitrary regular CW complex $K$ is homeomorphic to the geometric realization of the order complex of some poset. In particular $|K|\cong|\text{Bd}(K)|\cong|\Delta(\mathcal{F}(K))|$ where $\text{Bd}(K)$ is the barycentric subdivision of $K$. 
\end{prop}

Kaneta-Yoshinaga embedding and the Asao-Izumihara isomorphism with Proposition \ref{Koz} provides a way to strengthen the extended results of Sazdanovic and Summers. 

\begin{defn}
    Let $K$ be a finite dimensional regular CW complex with finitely many cells in each dimension. Define $\mathcal{G}(K)$ as the graph of the Hasse diagram of $\widehat{\mathcal{F}}(K)$, the cell closure inclusion poset of $K$ with $\hat0$ and $\hat1$ adjoined.
\end{defn}

\begin{thm}\label{CW_thm}
    Let $K$ be a finite dimensional regular CW complex with finitely many cells in each dimension. Then there exists an isomorphism 
    \[EMH_{k+2,\text{rk}(\widehat{\mathcal{F}}(K))}(\hat0,\hat1)\cong H_k(\Delta(\mathcal{F}(K)))\cong MH_{k+2,\text{rk}(\widehat{\mathcal{F}}(K))}(\hat0,\hat1)\]
    where $\hat0,\hat1\in V(\mathcal{G}(K))$ are the appropriate vertices corresponding to $\hat0,\hat1\in\widehat{\mathcal{F}}(K)$.
\end{thm}

\begin{cor}
    Let $K$ now be a topological space which has a finite dimensional regular CW complex structure with finitely many cells in each dimension. Then
    $EMH_{k+2,\text{rk}(\widehat{\mathcal{F}}(K))}(\hat0,\hat1)\cong H_k(K)\cong MH_{k+2,\text{rk}(\widehat{\mathcal{F}}(K))}(\hat0,\hat1).$
\end{cor}

The importance of this theorem lies in the fact that we can relax the requirements given by Kaneta-Yoshinaga and Sazdanovic-Summers to embed the homology of a well-known object into the magnitude homology of a graph. We no longer need to restrict ourselves only to triangulable manifolds. Instead, we can now consider a wider class of topological spaces which can be equipped with a regular CW structure. Specifically, when looking for posets and graphs with $p$-torsion in their order and magnitude homologies, we can now consider the Moore spaces $M(G,n),$  see 
\cite{hatcher2005algebraic} for the details,

Because triangulations are difficult and tedious to compute, we were also restricted to concrete examples of posets (and graphs) with $p$-torsion in their order (and magnitude) homologies which were of rank $2$ (and length $4$). 
Finding a regular CW structure on a space is much easier, and thus gives us a way to find explicit examples of $2$ and $3$ torsion in posets (and graphs) of rank $3$ (and length $5$) by finding regular CW structures on the lens spaces. Therefore, in Table \ref{torsion_graphs} we provide a list of graphs which contain $3$- and $5$-torsion in their magnitude and Eulerian magnitude homologies.

\begin{table}
\centering
    \begin{tabular}{|c|c|c|}
        \hline
        \multicolumn{1}{|c|}{Top. Space $K$} & CW-Structure & $\mathcal{G}(K)$ \\
        \hline
        \multicolumn{1}{|c|}{$M(\mathbb{Z}_3,1)$} & 
        \begin{tabular}{c}
        \begin{tikzpicture}[scale=0.5]
        \node[circle, draw, fill=black, scale=0.2] (a) [label=west:$a_2$] at (-2,2) {};
        \node[circle, draw, fill=black, scale=0.2] (b) [label=east:$a_2$] at (2,2) {};
        \node[circle, draw, fill=black, scale=0.2] (c) [label=south:$a_2$] at (0,-1.41) {};

        \draw[fill=none](0,0.845299) circle (2.34);

        \node[circle, draw, fill=black, scale=0.2] (o) [label=east:$a_0$] at (0,0.845299) {};

        \node[circle, draw, fill=black, scale=0.2] (g) [label=north:$a_1$] at (0,3.1547005) {};
        \node[circle, draw, fill=black, scale=0.2] (h) [label=west:$a_1$] at (-2,-0.309) {};
        \node[circle, draw, fill=black, scale=0.2] (i) [label=east:$a_1$] at (2,-0.309) {};

        \draw (o) -- (a);
        \draw (o) -- (b);
        \draw (o) -- (c);
        \draw (o) -- (g);
        \draw (o) -- (h);
        \draw (o) -- (i);
    \end{tikzpicture}
        \end{tabular}&
        \begin{tabular}{c}
        \begin{tikzpicture}[scale=.7]
            \node[circle, draw, fill=black, scale=0.3] (1)  at (-2,2) {};
            \node[circle, draw, fill=black, scale=0.3] (0)  at (0,2) {};
            \node[circle, draw, fill=black, scale=0.3] (2)  at (2,2) {};
            \node[circle, draw, fill=black, scale=0.3] (12)  at (-3.5,4) {};
            \node[circle, draw, fill=black, scale=0.3] (01)  at (-2.5,4) {};
            \node[circle, draw, fill=black, scale=0.3] (06)  at (-1.5,4) {};
            \node[circle, draw, fill=black, scale=0.3] (02)  at (-.5,4) {};
            \node[circle, draw, fill=black, scale=0.3] (05)  at (.5,4) {};
            \node[circle, draw, fill=black, scale=0.3] (03)  at (1.5,4) {};
            \node[circle, draw, fill=black, scale=0.3] (04)  at (2.5,4) {};
            \node[circle, draw, fill=black, scale=0.3] (23)  at (3.5,4) {};
            \node[circle, draw, fill=black, scale=0.3] (506)  at (-2.5,6) {};
            \node[circle, draw, fill=black, scale=0.3] (203)  at (-1.5,6) {};
            \node[circle, draw, fill=black, scale=0.3] (304)  at (-.5,6) {};
            \node[circle, draw, fill=black, scale=0.3] (405)  at (.5,6) {};
            \node[circle, draw, fill=black, scale=0.3] (102)  at (1.5,6) {};
            \node[circle, draw, fill=black, scale=0.3] (106)  at (2.5,6) {};

            \node[circle, draw, fill=black, scale=0.3] (top) at (0,7) {};

            \node[circle, draw, fill=black, scale=0.3] (bottom) at (0,1) {};

            \draw (0) -- (01) -- (102);
            \draw (0) -- (02) -- (203);
            \draw (0) -- (03) -- (304);
            \draw (0) -- (04) -- (405);
            \draw (0) -- (05) -- (506);
            \draw (0) -- (06) -- (106);
            \draw (1) -- (01) -- (106);
            \draw (1) -- (03) -- (203);
            \draw (1) -- (05) -- (405);
            \draw (1) -- (12) -- (102);
            \draw (1) -- (23) -- (203);
            \draw (2) -- (02) -- (102);
            \draw (2) -- (04) -- (304);
            \draw (2) -- (06) -- (506);
            \draw (2) -- (12) -- (304);
            \draw (2) -- (23) -- (405);
            \draw (12) -- (506);
            \draw (23) -- (106);
            \draw (bottom) -- (0);
            \draw (bottom) -- (1);
            \draw (bottom) -- (2);
            \draw (top) -- (102);
            \draw (top) -- (203);
            \draw (top) -- (304);
            \draw (top) -- (405);
            \draw (top) -- (506);
            \draw (top) -- (106);
        \end{tikzpicture} 
        \end{tabular}
           \\
        \hline
        \multicolumn{1}{|c|}{$M(\mathbb{Z}_5,1)$} & 
        \begin{tabular}{c}
        \begin{tikzpicture}[scale=1]
            \node[circle, draw, fill=black, scale=0.3] (a) [label=north:$a_1$] at (0,1) {};
            \node[circle, draw, fill=black, scale=0.2] (b) [label=west:$a_2$] at (-0.587786,0.809017) {};
            \node[circle, draw, fill=black, scale=0.3] (c) [label=west:$a_1$] at (-0.951057,0.309017) {};
            \node[circle, draw, fill=black, scale=0.3] (d) [label=west:$a_2$] at (-0.951057,-0.309017) {};
            \node[circle, draw, fill=black, scale=0.3] (e) [label=west:$a_1$] at (-0.587786,-0.809017) {};
            \node[circle, draw, fill=black, scale=0.3] (f) [label=south:$a_2$] at (0,-1) {};
            \node[circle, draw, fill=black, scale=0.3] (g) [label=east:$a_2$] at (0.587786,0.809017) {};
            \node[circle, draw, fill=black, scale=0.3] (h) [label=east:$a_1$] at (0.951057,0.309017) {};
            \node[circle, draw, fill=black, scale=0.3] (i) [label=east:$a_2$] at (0.951057,-0.309017) {};
            \node[circle, draw, fill=black, scale=0.3] (j) [label=east:$a_1$] at (0.587786,-0.809017) {};

            \draw[fill=none](0,0) circle (1);

            \node[circle, draw, fill=black, scale=0.3] (o) [label=east:$a_0$] at (0,0) {};

            \draw (o) -- (a);
            \draw (o) -- (b);
            \draw (o) -- (c);
            \draw (o) -- (d);
            \draw (o) -- (e);
            \draw (o) -- (f);
            \draw (o) -- (g);
            \draw (o) -- (h);
            \draw (o) -- (i);
            \draw (o) -- (j);
        \end{tikzpicture}
        \end{tabular}&  
        \begin{tabular}{c}
             \begin{tikzpicture}[scale=0.7]
            \node[circle, draw, fill=black, scale=0.3] (bottom)  at (0,1) {};

            \node[circle, draw, fill=black, scale=0.3] (1)  at (-2,2) {};
            \node[circle, draw, fill=black, scale=0.3] (0)  at (0,2) {};
            \node[circle, draw, fill=black, scale=0.3] (2)  at (2,2) {};

            \node[circle, draw, fill=black, scale=0.3] (12)  at (-3.85,4) {};
            \node[circle, draw, fill=black, scale=0.3] (01)  at (-3.15,4) {};
            \node[circle, draw, fill=black, scale=0.3] (02)  at (-2.45,4) {};
            \node[circle, draw, fill=black, scale=0.3] (03)  at (-1.75,4) {};
            \node[circle, draw, fill=black, scale=0.3] (04)  at (-1.05,4) {};
            \node[circle, draw, fill=black, scale=0.3] (05)  at (-.35,4) {};
            \node[circle, draw, fill=black, scale=0.3] (06)  at (.35,4) {};
            \node[circle, draw, fill=black, scale=0.3] (07)  at (1.05,4) {};
            \node[circle, draw, fill=black, scale=0.3] (08)  at (1.75,4) {};
            \node[circle, draw, fill=black, scale=0.3] (09)  at (2.45,4) {};
            \node[circle, draw, fill=black, scale=0.3] (010)  at (3.15,4) {};
            \node[circle, draw, fill=black, scale=0.3] (23)  at (3.85,4) {};

            \node[circle, draw, fill=black, scale=0.3] (102)  at (-3.15,6) {};
            \node[circle, draw, fill=black, scale=0.3] (203)  at (-2.45,6) {};
            \node[circle, draw, fill=black, scale=0.3] (304)  at (-1.75,6) {};
            \node[circle, draw, fill=black, scale=0.3] (405)  at (-1.05,6) {};
            \node[circle, draw, fill=black, scale=0.3] (506)  at (-.35,6) {};
            \node[circle, draw, fill=black, scale=0.3] (607)  at (.35,6) {};
            \node[circle, draw, fill=black, scale=0.3] (708)  at (1.05,6) {};
            \node[circle, draw, fill=black, scale=0.3] (809)  at (1.75,6) {};
            \node[circle, draw, fill=black, scale=0.3] (9010)  at (2.45,6) {};
            \node[circle, draw, fill=black, scale=0.3] (1010)  at (3.15,6) {};

            \node[circle, draw, fill=black, scale=0.3] (top)  at (0,7) {};

            \draw (0) -- (01) -- (102);
            \draw (0) -- (02) -- (203);
            \draw (0) -- (03) -- (304);
            \draw (0) -- (04) -- (405);
            \draw (0) -- (05) -- (506);
            \draw (0) -- (06) -- (607);
            \draw (0) -- (07) -- (708);
            \draw (0) -- (08) -- (809);
            \draw (0) -- (09) -- (9010);
            \draw (0) -- (010) -- (1010);
            \draw (1) -- (01) -- (1010);
            \draw (1) -- (03) -- (203);
            \draw (1) -- (05) -- (405);
            \draw (1) -- (07) -- (607);
            \draw (1) -- (09) -- (809);
            \draw (1) -- (12) -- (102);
            \draw (1) -- (23) -- (203);
            \draw (2) -- (02) -- (102);
            \draw (2) -- (04) -- (304);
            \draw (2) -- (06) -- (506);
            \draw (2) -- (08) -- (708);
            \draw (2) -- (010) -- (9010);
            \draw (2) -- (12) -- (304);
            \draw (2) -- (23) -- (405);
            \draw (12) -- (506);
            \draw (12) -- (708);
            \draw (12) -- (9010);
            \draw (23) -- (607);
            \draw (23) -- (809);
            \draw (23) -- (1010);
            \draw (bottom) -- (0);
            \draw (bottom) -- (1);
            \draw (bottom) -- (2);
            \draw (top) -- (102);
            \draw (top) -- (203);
            \draw (top) -- (304);
            \draw (top) -- (405);
            \draw (top) -- (506);
            \draw (top) -- (607);
            \draw (top) -- (708);
            \draw (top) -- (809);
            \draw (top) -- (9010);
            \draw (top) -- (1010);
        \end{tikzpicture}
        \end{tabular}\\
        \hline
        \multicolumn{1}{|c|}{$L(3,1)$}    & 
        \begin{tabular}{c}
             \begin{tikzpicture}[scale=0.6]
        \shade[ball color = gray!40, opacity = 0.4] (0,0) circle (2.3cm);
        \draw (0,0) circle (2.3cm);
        \draw[name path=equator] (-2.3,0) arc (180:360:2.3 and 0.6);
        \draw[name path=arc_1, rotate=90] (-2.3,0) arc (180:360:2.3 and 1.5);
        \draw[name path=arc_2, rotate=-90] (-2.3,0) arc (180:360:2.3 and 1.8);
        \draw[name path=arc_4, dashed, rotate=90] (-2.3,0) arc (180:360:2.3 and .3);
        \draw[name path=equator_2, dashed] (2.3,0) arc (0:180:2.3 and 0.6);
        \draw[name path=arc_3, rotate=-90] (-2.3,0) arc (180:360:2.3 and .3);
        \draw[name path=arc_5, dashed, rotate=90] (-2.3,0) arc (180:360:2.3 and 1.8);
        \draw[name path=arc_6, dashed, rotate=-90] (-2.3,0) arc (180:360:2.3 and 1.5);

        \node[circle, draw, fill=black, scale=0.3] (0) [label=north:$a_0$] at (0,0) {};
        \node[circle, draw, fill=black, scale=0.3] (7) [label=north:$a_3$]  at (0,2.3) {};
        \node[circle, draw, fill=black, scale=0.3] (8) [label=south:$a_3$] at (0,-2.3) {};
        
        \fill[black, opacity=1, below right, name intersections={of=equator and arc_1}] (intersection-1) circle (1.7pt) node {$a_1$};
        \node[coordinate, name intersections = {of = equator and arc_1}] (1) at  (intersection-1) {};
        \fill[black, opacity=1, below left, name intersections={of=equator and arc_2}] (intersection-1) circle (1.7pt) node {$a_1$};
        \node[coordinate, name intersections = {of = equator and arc_2}] (2) at  (intersection-1) {};
        \fill[black, opacity=1, below right, name intersections={of=equator and arc_3}] (intersection-1) circle (1.7pt) node {$a_2$};
        \node[coordinate, name intersections = {of = equator and arc_3}] (3) at  (intersection-1) {};

        \fill[black, opacity=1, below right, name intersections={of=equator_2 and arc_4}] (intersection-1) circle (1.7pt) node {$a_1$};
        \node[coordinate, name intersections = {of = equator_2 and arc_4}] (4) at  (intersection-1) {};
        \fill[black, opacity=1, below right, name intersections={of=equator_2 and arc_5}] (intersection-1) circle (1.7pt) node {$a_2$};
        \node[coordinate, name intersections = {of = equator_2 and arc_5}] (5) at  (intersection-1) {};
        \fill[black, opacity=1, below right, name intersections={of=equator_2 and arc_6}] (intersection-1) circle (1.7pt) node {$a_2$};
        \node[coordinate, name intersections = {of = equator_2 and arc_6}] (6) at  (intersection-1) {};

        \draw[dashed] (0) -- (1);
        \draw[dashed] (0) -- (2);
        \draw[dashed] (0) -- (3);
        \draw[dashed] (0) -- (4);
        \draw[dashed] (0) -- (5);
        \draw[dashed] (0) -- (6);
        \draw[dashed] (0) -- (7);
        \draw[dashed] (0) -- (8);
    \end{tikzpicture}
        \end{tabular}&
        \begin{tabular}{c}
             \begin{tikzpicture}[scale=0.45]
            \node[circle, draw, fill=black, scale=0.3] (bottom)  at (0,0) {};

            \node[circle, draw, fill=black, scale=0.3] (1)  at (-6,2) {};
            \node[circle, draw, fill=black, scale=0.3] (0)  at (-2,2) {};
            \node[circle, draw, fill=black, scale=0.3] (7) at (2,2) {};
            \node[circle, draw, fill=black, scale=0.3] (2) at (6,2) {};

            \node[circle, draw, fill=black, scale=0.3] (12)  at (-7.5,6) {};
            \node[circle, draw, fill=black, scale=0.3] (17)  at (-6.5,6) {};
            \node[circle, draw, fill=black, scale=0.3] (27)  at (-5.5,6) {};
            \node[circle, draw, fill=black, scale=0.3] (37)  at (-4.5,6) {};
            \node[circle, draw, fill=black, scale=0.3] (01)  at (-3.5,6) {};
            \node[circle, draw, fill=black, scale=0.3] (02)  at (-2.5,6) {};
            \node[circle, draw, fill=black, scale=0.3] (03)  at (-1.5,6) {};
            \node[circle, draw, fill=black, scale=0.3] (07)  at (-.5,6) {};
            \node[circle, draw, fill=black, scale=0.3] (08)  at (.5,6) {};
            \node[circle, draw, fill=black, scale=0.3] (04)  at (1.5,6) {};
            \node[circle, draw, fill=black, scale=0.3] (05)  at (2.5,6) {};
            \node[circle, draw, fill=black, scale=0.3] (06)  at (3.5,6) {};
            \node[circle, draw, fill=black, scale=0.3] (47)  at (4.5,6) {};
            \node[circle, draw, fill=black, scale=0.3] (57)  at (5.5,6) {};
            \node[circle, draw, fill=black, scale=0.3] (67)  at (6.5,6) {};
            \node[circle, draw, fill=black, scale=0.3] (23)  at (7.5,6) {};

            \node[circle, draw, fill=black, scale=0.3] (102)  at (-8.05,10) {};
            \node[circle, draw, fill=black, scale=0.3] (203)  at (-7.35,10) {};
            \node[circle, draw, fill=black, scale=0.3] (304)  at (-6.65,10) {};
            \node[circle, draw, fill=black, scale=0.3] (407)  at (-5.95,10) {};
            \node[circle, draw, fill=black, scale=0.3] (507)  at (-5.25,10) {};
            \node[circle, draw, fill=black, scale=0.3] (607)  at (-4.55,10) {};
            \node[circle, draw, fill=black, scale=0.3] (108)  at (-3.85,10) {};
            \node[circle, draw, fill=black, scale=0.3] (208)  at (-3.15,10) {};
            \node[circle, draw, fill=black, scale=0.3] (308)  at (-2.45,10) {};
            \node[circle, draw, fill=black, scale=0.3] (127)  at (-1.75,10) {};
            \node[circle, draw, fill=black, scale=0.3] (237)  at (-1.05,10) {};
            \node[circle, draw, fill=black, scale=0.3] (347)  at (-.35,10) {};
            \node[circle, draw, fill=black, scale=0.3] (457)  at (.35,10) {};
            \node[circle, draw, fill=black, scale=0.3] (567)  at (1.05,10) {};
            \node[circle, draw, fill=black, scale=0.3] (167)  at (1.75,10) {};
            \node[circle, draw, fill=black, scale=0.3] (408)  at (2.45,10) {};
            \node[circle, draw, fill=black, scale=0.3] (508)  at (3.15,10) {};
            \node[circle, draw, fill=black, scale=0.3] (608)  at (3.85,10) {};
            \node[circle, draw, fill=black, scale=0.3] (107)  at (4.55,10) {};
            \node[circle, draw, fill=black, scale=0.3] (207)  at (5.25,10) {};
            \node[circle, draw, fill=black, scale=0.3] (307)  at (5.95,10) {};
            \node[circle, draw, fill=black, scale=0.3] (405)  at (6.65,10) {};
            \node[circle, draw, fill=black, scale=0.3] (506)  at (7.35,10) {};
            \node[circle, draw, fill=black, scale=0.3] (106)  at (8.05,10) {};

            \node[circle, draw, fill=black, scale=0.3] (1027)  at (-7.15,14) {};
            \node[circle, draw, fill=black, scale=0.3] (2037)  at (-5.85,14) {};
            \node[circle, draw, fill=black, scale=0.3] (3047)  at (-4.55,14) {};
            \node[circle, draw, fill=black, scale=0.3] (4058)  at (-3.25,14) {};
            \node[circle, draw, fill=black, scale=0.3] (5068)  at (-1.95,14) {};
            \node[circle, draw, fill=black, scale=0.3] (1068)  at (-.65,14) {};
            \node[circle, draw, fill=black, scale=0.3] (1028)  at (.65,14) {};
            \node[circle, draw, fill=black, scale=0.3] (2038)  at (1.95,14) {};
            \node[circle, draw, fill=black, scale=0.3] (3048)  at (3.25,14) {};
            \node[circle, draw, fill=black, scale=0.3] (4057)  at (4.55,14) {};
            \node[circle, draw, fill=black, scale=0.3] (5067)  at (5.85,14) {};
            \node[circle, draw, fill=black, scale=0.3] (1067)  at (7.15,14) {};

            \node[circle, draw, fill=black, scale=0.3] (top)  at (0,16) {};

            \draw (0) -- (01) -- (102) -- (1027);
            \draw (0) -- (02) -- (203) -- (2037);
            \draw (0) -- (03) -- (304) -- (3047);
            \draw (0) -- (04) -- (405) -- (4057);
            \draw (0) -- (05) -- (506) -- (5067);
            \draw (0) -- (06) -- (106) -- (1067);
            \draw (0) -- (07);
            \draw (0) -- (08);
            \draw (1) -- (01) -- (107) -- (1027);
            \draw (1) -- (03) -- (307) -- (3047);
            \draw (1) -- (05) -- (507) -- (5067);
            \draw (1) -- (12);
            \draw (1) -- (23);
            \draw (1) -- (17) -- (107) -- (1067);
            \draw (1) -- (37) -- (307) -- (2037);
            \draw (1) -- (57) -- (507) -- (4057);
            \draw (2) -- (02) -- (207) -- (1027);
            \draw (2) -- (04) -- (407) -- (3047);
            \draw (2) -- (06) -- (607) -- (5067);
            \draw (2) -- (12);
            \draw (2) -- (23);
            \draw (2) -- (27) -- (207) -- (2037);
            \draw (2) -- (47) -- (407) -- (4057);
            \draw (2) -- (67) -- (607) -- (1067);
            \draw (7) -- (07);
            \draw (7) -- (08);
            \draw (7) -- (17) -- (108) -- (1028);
            \draw (7) -- (27) -- (208) -- (2038);
            \draw (7) -- (37) -- (308) -- (3048);
            \draw (7) -- (47) -- (408) -- (4058);
            \draw (7) -- (57) -- (508) -- (5068);
            \draw (7) -- (67) -- (608) -- (1068);
            \draw (01) -- (108) -- (1068);
            \draw (02) -- (208) -- (1028);
            \draw (03) -- (308) -- (2038);
            \draw (04) -- (408) -- (3048);
            \draw (05) -- (508) -- (4058);
            \draw (06) -- (608) -- (5068);
            \draw (07) -- (107);
            \draw (07) -- (207);
            \draw (07) -- (307);
            \draw (07) -- (407);
            \draw (07) -- (507);
            \draw (07) -- (607);
            \draw (08) -- (108);
            \draw (08) -- (208);
            \draw (08) -- (308);
            \draw (08) -- (408);
            \draw (08) -- (508);
            \draw (08) -- (608);
            \draw (12) -- (102) -- (1028);
            \draw (12) -- (304) -- (3048);
            \draw (12) -- (506) -- (5068);
            \draw (12) -- (127) -- (1027);
            \draw (12) -- (347) -- (3047);
            \draw (12) -- (567) -- (5067);
            \draw (23) -- (203) -- (2038);
            \draw (23) -- (405) -- (4058);
            \draw (23) -- (106) -- (1068);
            \draw (23) -- (237) -- (2037);
            \draw (23) -- (457) -- (4057);
            \draw (23) -- (167) -- (1067);
            \draw (17) -- (308);
            \draw (17) -- (127);
            \draw (17) -- (167);
            \draw (27) -- (408);
            \draw (27) -- (127) -- (3048);
            \draw (27) -- (237);
            \draw (37) -- (508);
            \draw (37) -- (237) -- (4058);
            \draw (37) -- (347);
            \draw (47) -- (608);
            \draw (47) -- (347) -- (5068);
            \draw (47) -- (457);
            \draw (57) -- (108);
            \draw (57) -- (457) -- (1068);
            \draw (57) -- (567) -- (1028);
            \draw (67) -- (208);
            \draw (67) -- (567);
            \draw (67) -- (167) -- (2038);

            \draw (bottom) -- (0);
            \draw (bottom) -- (1);
            \draw (bottom) -- (2);
            \draw (bottom) -- (7);

            \draw (top) -- (1027);
            \draw (top) -- (2037);
            \draw (top) -- (3047);
            \draw (top) -- (4057);
            \draw (top) -- (5067);
            \draw (top) -- (1067);
            \draw (top) -- (1028);
            \draw (top) -- (2038);
            \draw (top) -- (3048);
            \draw (top) -- (4058);
            \draw (top) -- (5068);
            \draw (top) -- (1068);
        \end{tikzpicture}
        \end{tabular}\\
        \hline
    \end{tabular}
     \caption{Graphs with torsion in their magnitude homologies}
    \label{torsion_graphs}
\end{table}


\section{Magnitude Homology: computations and relations}\label{meaning}

One of the benefits of Eulerian magnitude homology is that it has finite support
unlike magnitude homology which can contain infinitely many nontrivial homology groups. This is a consequence of the fact that as $k$ and $\ell$ both increase, redundant paths such as $(v,w,v,w,v,w,...,v,w)$ and the like completely dominate the $(k,\ell)$-magnitude chain groups, and while they are always in the kernel of a boundary map they are not necessarily in the image. Magnitude homology for trees provides a great example of this phenomena. However, it is natural to expect that  no Eulerian $k$-paths  can exist in Eulerian magnitude homology for $k>|V|-1$. Additionally, $\ell$ has to bounded as well. Let $k_{max}=|V|-1$ in the rest of the section.
\begin{defn}
    Given a graph $G$, the \textbf{maximum Eulerian length} of the graph is the value $\ell_{\textrm{max}}:=\max\{\textrm{len}(v):v\in ET(G)\}$ where $ET(G)$ is the set of all Eulerian paths in $G$.
\end{defn}
\begin{prop}
    For $0\leq k\leq k_{\textrm{max}}$ let $ET_k(G)$ be the set of all $k$-paths of $G$. Then $\ell_{\textrm{max}}=\max\{\textrm{len}(v):v\in ET_{k_{\textrm{max}}}(G)\}$.
\end{prop}
\begin{proof}
    Assume that $\ell_{\text{max}}$ was the length of some $k'$-path $v$ for $k'<k_{\text{max}}$. Then there exists at least one more vertex which can be added to the end of $v$ to make it a $(k'+1)$-path which has a length greater than $\ell_{\text{max}}$, which is a contradiction.
\end{proof}
Hence, when searching for a potential $\ell_{max}$, one only needs to find the maximum length over all $k_{max}$-paths. It is always guaranteed that such maximums exist since we always have that $|V(G)|$ and $|E(G)|$ are finite. Now another pleasant property of Eulrian magnitude homology is that we can find ``the end" of the support for the homology groups. 
\begin{thm}
    For a graph $G$ let $EMC_\text{max}(G)=EMC_{k_\text{max},\ell_\text{max}}(G)$ and similarly let $EMH_{\text{max}}(G)=EMH_{k_{\text{max}},\ell_{\text{max}}}(G)$. Then $EMH_{\text{max}}(G)$ is always nontrivial, and in particular $EMH_\text{max}(G)=EMC_\text{max}(G)$. 
\end{thm}
\begin{proof}
    Since $EMH_{\ast,\ast}(G)$ is supported by all $k$-paths of length $\ell$ for $k\leq k_{\text{max}}$ and $\ell\leq\ell_{\text{max}}$, there cannot exist any $k'$-paths of length $\ell'$ for $k'>k_\text{max}$ and $\ell'>\ell_\text{max}$. Hence, $EMH_\text{max}(G)=\ker\partial_{k_\text{max},\ell_\text{max}}$. Now we only need to show that there cannot exist any $(k_\text{max}-1)$-paths of length $\ell_\text{max}$. Similar to the previous proof, we assume that there can exist a $(k_\text{max}-1)$-path $v$ of length $\ell_\text{max}$. This implies that we can add a vertex to the end of $v$ to make a $k_\text{max}$-path of length greater than $\ell_\text{max}$, which is a contradiction to the definition of $\ell_\text{max}$. Hence, $EMC_{k_\text{max}-1,\ell_\text{max}}(G)=0$ so that $EMC_\text{max}(G)=\ker\partial_{k_\text{max},\ell_{\text{max}}}$ and $EMC_\text{max}$ is nontrivial. This completes the proof.
\end{proof}
As a corollary,  there always exists at least one $\ell$ value, denoted by $\ell_{\text{max}}$, such that $EMH_{k_{\text{max}},\ell}(G)$ is nontrivial  and the maximum length of a paths is $k_{\text{max}}$. 

Next, we analyze discriminant magnitude homology for several small values of $k$ and $\ell$, and provide computational results for magnitude homology groups of some classes of graphs. 

\begin{prop}
    For a graph $G$, $DMH_{0,0}(G)=0=DMH_{1,1}(G)$, and $DMH_{2,2}(G)\cong MH_{2,2}(G)$ if and only if $EMH_{2,2}(G)=0$.
\end{prop}
\begin{proof}
    Note that $EMC_{0,0}(G)=\Z\langle V(G)\rangle =MC_{0,0}(G)$ and $EMC_{1,1}(G)=\Z\langle E(G)\rangle=MC_{1,1}(G)$. Thus $DMC_{0,0}(G)=0=DMC_{1,1}(G)$ which implies the first result. Now $EMH_{1,2}(G)$ is always trivial in the same sense that $MH_{1,2}(G)$ is always trivial: the kernel of $\partial_{1,2}$ is equal to all of $EMC_{1,2}(G)$, but the image of $\partial_{2,2}$ are exactly all the pairs of vertices such that their distance apart from each other is $2$. Hence $\ker\partial_{1,2}=\im\partial_{2,2}$. Thus, if $EMH_{2,2}(G)$ is trivial, we get from the long exact sequence in magnitude homology
    \[0\to MH_{2,2}(G)\to DMH_{2,2}(G)\to 0.\]
\end{proof}

\begin{defn}
    A graph $G$ has \text{diagonal magnitude homology} if for $k\geq0$, $MH_{k,k}(G)$ are the only nontrivial magnitude homology groups of $G$.
\end{defn}

\begin{lem}\label{dmh_lemma}
    Let $G$ have diagonal magnitude homology. If $EMH_{k-1,k}(G)$ contains no torsion subgroups and $EMH_{k,k}(G)=0$, then
    \[\text{rank}(DMH_{k,k}(G))=\text{rank}(MH_{k,k}(G))+\text{rank}(EMH_{k-1,k}(G))\]
    and $DMH_{k,\ell}(G)=0$ for all $k\neq\ell$.
\end{lem}
\begin{proof}
    This follows immediately from the long exact sequence for magnitude homology. If $EMH_{k,k}(G)$ is trivial, then we get the following short exact sequence
    \[0\to MH_{k,k}(G)\to DMH_{k,k}(G)\to EMH_{k-1,k}(G)\to 0\]
    which splits because $EMH_{k-1,k}(G)$ is a free abelian group, and hence a projective $\Z$ module. The long exact sequence also yields that $DMH_{k,\ell}(G)=0$ if and only if the map $\iota_\ast:EMH_{k,\ell}(G)\to MH_{k,\ell}(G)$ is surjective, which it surely is for all $k\neq\ell$ since $G$ has diagonal magnitude homology.  
\end{proof}

Note that the diagonality of the graph with respect to magnitude homology does not imply that  nontrivial Eulerian homology groups along the diagonal. Giusti and Menara provide a method for finding  generators of the Eulerian homology groups along the diagonal \cite{giusti2024eulerian}. We use their result to Eulerian magnitude of trees.

\begin{prop}\label{tree_diag}
    For a tree $T$, $EMH_{k,k}(T)=0$ for all $k\geq 2$.
\end{prop}
\begin{proof}
    Giusti and Menara proved that a $k$-path $v=(v_0,...,v_k)\in EMC_{k,k}(G)$ for a graph $G$ has zero differential $\partial^i_{k,k}$ if and only if there exists an edge $\{v_{i-1},v_{i+1}\}$ in the graph $G$ \cite{giusti2024eulerian}. However, this can never occur in a tree $T$ since trees can never contain subgraphs isomorphic to a cycle graph. Hence, for $v\in EMC_{k,k}(T)$, $\partial_{k,k}^i(v)\neq 0$ for all $1\leq i\leq k-1$. Additionally, for any pair of vertices $a,b\in V(T)$ there either exists a unique $k$-path of length $k$ from $a$ to $b$ or there exists no $k$-paths of length $k$ from $a$ to $b$. Thus, the group $EMC_{k,k}(a,b)$ only consists of one generator which has nontrivial boundary, and therefore, $\ker\partial_{k,k}=0$ as desired. 
\end{proof}

Combining Lemma  \ref{dmh_lemma} with some of the  abundance of results on magnitude homology showing that trees are diagonal \cite{hepworth2015categorifying, gu2018graph, asao2020geometric} we get the following result.

\begin{cor}\label{tree_corollary}
    For a tree $T$, if $EMH_{2,3}(T)$ and $EMH_{3,4}(T)$ are free, then 
    \[\text{rank}(DMH_{k,k}(T))=\text{rank}(MH_{k,k}(T))+\text{rank}(EMH_{k-1,k}(T))\]
    for $k=3,4$ and $DMH_{k,k}(T)\cong MH_{k,k}(T)$ for $k=2$ and $k\geq 5$, while $DMH_{k,\ell}(T)=0$ for all $k\neq\ell$. 
\end{cor}

We now use these results and the long exact sequence in magnitude homology to help us complete a list of all the ranks of the Eulerian, discriminant, and ordinary magnitude homology groups for star trees. 
\begin{defn}
    Given a positive integer $n$, the \textbf{star tree $S_n$} is the tree $S_n=([n]_0,\{\{0,i\}_{i\in[n]}\})$.
\end{defn}
\begin{prop}  Maximal possible nontrivial magnitude homology of a star tree $S_n$ is supported in second grading 
 $\ell_{\text{max}}=2n-1$.
\end{prop}
\begin{proof}
    The distance between any two vertices $0\neq i\neq j\neq 0$ will be $2$ while $d(0,i)=1$ for all $i\in[n]$. Since vertices cannot repeat, the vertex $0$ is either placed at one of the ends of a $k_{\text{max}}$-path or in between the ends. If $0$ is placed in between the ends the length of such a $k_{\text{max}}$-path will be $2n-2$. If $0$ is placed at either of the ends of such a $k_{\text{max}}$-path then the length of such a path will be $2n-1$. Since these exhaust the possible lengths of all possible $k_{\text{max}}$-paths in $S_n$, we get $\ell_{\text{max}}=2n-1$.
\end{proof}

\begin{thm} The Eulerian Magnitude homology groups of star trees are torsion-free and the ranks are determined by the following formula: 
    \[\text{rank}(EMH_{k,\ell}(S_n))=\begin{cases}
    |V| & k=\ell=0 \\ 
    2|E| & k=\ell=1 \\
    2\cdot n\downarrow_k & \ell=2k-1,\,2\leq k\leq n \\
    (k-2)\cdot n\downarrow_k & \ell=2(k-1),\, 3\leq k\leq n \\ 
    0 & \text{otherwise.}
\end{cases}\]
\end{thm}
\begin{proof}
   For any vertex $a\in [n_0]$ $d(a,0)=1$ while $d(a,b)=2$ for any $0\neq b\in [n]_0$. Since vertices cannot repeat in an Eulerian $k$-path, $0$ can only appear at most once in any $k$-path. Thus, we have the following types of $k$-paths and their lengths
    \begin{eqnarray}
        \text{len}(0,v_1,...,v_k)=2k-1=\text{len}(v_1,...,v_k,0), & & 2\leq k\leq n \label{star1}\\
        \text{len}(v_1,...,0,...,v_k)=2(k-1), & & 3\leq k\leq n \label{star2}\\
        \text{len}(v_0,v_1,...,v_k)=2k, & & 2\leq k\leq n-1,\; v_i\neq 0 \label{star3}
    \end{eqnarray}
    where Equations \ref{star1} and \ref{star2} also hold for $n$-paths, but Equation \ref{star3} cannot. For any $k$-path of length $2k-1$, there cannot exist a $k+1$-path of length $2k-1$ so that $\text{Im}\,\partial_{k+1,2k-1}=0$ and consequently $EMH_{k,2k-1}(S_n)=\ker\partial_{k,2k-1}$. Since the removal of any vertex other than $0$ decreases the length of a path, we have 
    \[\partial_{k,2k-1}(0,v_1,...,v_k)=0=\partial_{k,2k-1}(v_1,...,v_k,0).\]
    Thus finding the rank of $EMH_{k,2k-1}(S_n)$ amounts to counting all such Eulerian $k$-paths of type in Equation \ref{star1}. Recall that each vertex can only be included once in each Eulerian $k$-path. Then counting the Eulerian $k$-paths of type in Equation \ref{star1} is the same as counting two times the number of permutations of $[n]$ of length $k$. Therefore, 
    \[\text{rank}(EMH_{k,2k-1}(S_n))=\text{rank}(\partial_{k,2k-1})=2\cdot n\downarrow_k.\]
    For any $k$-path of length $2(k-1)$, we have
    \[\partial_{k,2(k-1)}(v_1,...,0,...,v_k)=(-1)^i(v_1,...,v_k)\]
    where $0$ is in the $i$th position for $1<i<k$. Again, there cannot exist any $k+1$-paths of length $2(k-1)$ so $\text{Im}\,\partial_{k+1,2(k-1)}=0$, which again yields $EMH_{k,2(k-1)}(S_n)=\ker\partial_{k,2(k-1)}$. Now we can compute the rank of $\ker\partial_{k,2(k-1)}$ in the following way. For each $k-1$-path $v=(v_1,...,v_k)$ of length $2(k-1)$ there are $k-1$ possible $k$-paths of the same length which have $v$ as a boundary with the appropriate sign. Hence, a generator of $\ker\partial_{k,2(k-1)}$ must be some sum of such $k$-paths. To find how many such sums there are, we can consider the homomorphism $\varphi:\Z^{k-1}\to \Z$ which sends $(0,...,1,...,0)\mapsto(-1)^i$ for $1$ in the $i$th position in the tuple. Here $\Z^{k-1}\cong\Z\langle(v_1,...,0,...,v_k)\rangle$ and $\Z\cong\Z\langle(v_1,...,v_k)\rangle$ so that $\varphi$ mimics the way that $k$-paths of the form $(v_1,...,0,...,v_k)$ get sent to the $k-1$-path $(v_1,...,v_k)$ under the boundary homomorphism $\partial_{k,2(k-1)}$. Thus, the rank of $\ker\varphi$ yields the number of generators in $\ker\partial_{k,2(k-1)}$ which are sums of $k$-paths of the form $(v_1,...,0,..,v_k)$ where each $v_i$ are fixed and the position of $0$ varies between $v_1$ and $v_k$. The rank of $\ker\varphi$ can be easily computed using the fact that $\im\varphi\cong\Z$ is a projective $\Z$-module. Hence, the short exact sequence
    \[0\to\ker\varphi\to\Z^{k-1}\to\text{Im}\,\varphi\to 0\]
    splits, which implies $\text{rank}(\ker\varphi)=k-2$. This yields that the rank of $\ker\partial_{k,2(k-1)}$ is $\text{rank}(\ker\varphi)$ times the number of $k-1$-paths $(v_1,...,v_k)$, and this is the same as $k-2$ times the number of permutations of $[n]$ of length $k$. Therefore
    \[\text{rank}(EMH_{k,2(k-1)}(S_n))=\text{rank}(\ker\partial_{k,2(k-1)})=(k-2)\cdot n\downarrow_k.\]
    The only exception to the above, is when $k=2$ and $\ell=2$. But $S_n$ is a tree so that $EMH_{2,2}(S_n)$ must be trivial according to Proposition \ref{tree_diag}. Finally, for any $k$-path of length $2k$, $\partial_{k,2k}(v_0,v_1,...,v_k)=0$ so that $\ker\partial_{k,2k}=EMC_{k,2k}(S_n)$. However, $\partial_{k+1,2k}(v_1,...,0,...,v_k)=(-1)^i(v_1,...,v_k)$ where $0$ is in the $i$th position in the $k+1$-path, so that $\text{Im}\,\partial_{k+1,2k}=EMC_{k,2k}(S_n)$. Thus, $EMH_{k,2k}(S_n)=0$.
\end{proof}

Giusti and Menara  show that $EMH_{k,k}(G)=0$ for $k\geq 2$ then $DMH_{k,k}(G)\cong MH_{k,k}(G)$ for all $k\geq 5$. Since the Eulerian magnitude homology groups of the star trees are torsion-free, we can use Lemma \ref{dmh_lemma} to compute the rest. 

\begin{cor}
    The discriminant magnitude homology groups of star trees are torsion-free and the ranks are determined by the following formulas: 
    \[\text{rank}(DMH_{k,\ell}(S_n))=\begin{cases}
        2(n+n\downarrow_2) & k=\ell=3 \\
        2n+n\downarrow_3 & k=\ell=4 \\
        2n & k=\ell=2,\,k=\ell\geq 5 \\
        0 & \text{otherwise.}
    \end{cases}\]
\end{cor}

Now we turn our attention to answering the question: are there graphs which have diagonal Eulerian magnitude homology? There are many different types of graphs which exhibit diagonal magnitude homology such as outerplanar graphs \cite{sazdanovic2021torsion}, pawful graphs \cite{gu2018graph}, trees \cite{hepworth2015categorifying, gu2018graph, asao2020geometric}, and the complete graphs \cite{hepworth2015categorifying}. However, after computing the Eulerian magnitude homology of all of these types, only one yielded a support only along the diagonal. 

\begin{thm}
    For a graph $G=(V,E)$, $\ell_\text{max}=k_\text{max}$ if and only if $G$ is a complete graph $K_n$. Furthermore, the Eulerian magnitude homology groups of $K_n$ are torsion free with $\text{rank}(EMH_{k,k}(K_n))=n\downarrow_k$.
\end{thm}
\begin{proof}
    Giusti and Menara \cite{giusti2024eulerian} state that walks around cliques always have zero differential. Since every $k$-path in a complete graph $K_n$ cannot have any other length besides $k$, we get $EMC_{k,\ell}(K_n)$ is trivial for all $\ell\neq k$. Therefore, $K_n$ has nontrivial Eulerian magnitude homology only on its diagonal, which implies that $EMH_\text{max}(K_n)$ is on the diagonal and thus $k_\text{max}=\ell_\text{max}$. Conversely, assume that for a graph $G$, $\ell_\text{max}=k_\text{max}$. Then every $k_\text{max}$-path must be of length $k_\text{max}$. Hence, for every pair of vertices $a,b\in V$ we must have $\{a,b\}\in E$. Therefore, $G=K_{k_\text{max} +1}$. 

    Now for a complete graph $K_n$, since $EMC_{k,\ell}(K_n)=0$ for all $\ell\neq k$, we must have $EMC_{k,k}(K_n)=EMH_{k,k}(K_n)$. Thus, counting the rank of $EMH_{k,k}(K_n)$ amounts to counting all the Eulerian $k$-paths of $K_n$ which is the number of permutations of $[n]$ of length $k$. 
\end{proof}

\begin{cor}\label{Euler_diag}
    A graph $G$ has diagonal Eulerian magnitude homology if and only if $G$ is a complete graph.
\end{cor}
\begin{proof} Notice the following sequence of equivalences:  $G$  is diagonal if and only  $\ell_\text{max}=k_\text{max}$ if and only if $G=K_{k_\text{max}+1.}$ 
\end{proof}

This result is analogous to the fact that graphs which are diagonal in magnitude homology 
have girth $\text{gir}(G)=3,4,$ or infnity.  
 \cite{asao2024girth}.

\begin{cor}
    The discriminant magnitude homology groups of the complete graphs are supported only on the diagonal and are torsion-free with the ranks determined by the formula 
    \[\text{rank}(DMH_{k,k}(K_n))=n(n-1)^k-n\downarrow_k\] for $k\geq 2$. 
\end{cor}
\begin{proof}
    Since $MC_{k,\ell}(K_n)=\varnothing=EMC_{k,\ell}(K_n)$ for all $k\neq\ell$, it follows that $DMC_{k,\ell}(K_n)=\varnothing$ for all $k\neq\ell$. Hence, the discriminant magnitude homology of $K_n$ must be supported only along the diagonal and will thus also be torsion-free. The long exact sequence for magnitude homology yields
    \[0\to EMH_{k,k}(K_n)\to MH_{k,k}(K_n)\to DMH_{k,k}(K_n)\to0\]
    which splits since $DMH_{k,k}(K_n)$ is free abelian. Therefore, 
    \[\text{rank}(DMH_{k,k}(K_n))=\text{rank}(MH_{k,k}(K_n))-\text{rank}(EMH_{k,k}(K_n)),\]
    which completes the proof.
\end{proof}

Corollary \ref{Euler_diag} implies that outerplanar graphs and pawful graphs do not have diagonal Eulerian magnitude homology and diagonality is not a property that transfers over from magnitude homology to Eulerian magnitude homology. In general, many of the nice properties exhibited by magnitude homology do not transfer over to Eulerian magnitude homology. Consider the star trees $S_4$ and $S_2$. Notice that $S_4$ can be seen as two $S_2$ trees wedged together at the vertex $0$. In ordinary magnitude homology we would have $MH_{k,\ell}(S_4)\cong MH_{k,\ell}(S_2)\oplus MH_{k,\ell}(S_2)$. However, we can easily see that this cannot be true for Eulerian magnitude homology since $\text{rank}(EMH_{2,3}(S_2))=4$ while $\text{rank}(EMH_{2,3}(S_4))=24$. In particular, this shows that the Mayer-Vietoris sequence does not hold in general for Eulerian magnitude homology. In Figures \ref{fig:whitney1} and \ref{fig:whitney2} we also exhibit graphs $H$ and $K$ which differ by a Whitney twist but present two varying Eulerian magnitude homologies; showing that although ordinary magnitude homology is invariant under Whitney twists, Eulerian magnitude homology is not.

\begin{figure}[h]
\centering
    \begin{subfigure}{0.38\textwidth}
    \begin{tabular}{c}
    \def\arraystretch{1.25}
        \begin{tabular}{|c|c|c|c|c|c|}
         \hline
           \diagbox{$\ell$}{$k$} & 0 & 1 & 2 & 3 & 4  \\
            \hline
            0 & $\Z^5$ &   &   &   &      \\
            \hline
            1 &   &$\Z^{10}$&   &   &       \\
            \hline
            2 &   &   & $\Z^6$ &   &       \\
            \hline
            3 &   &   & $\Z^{20}$ &   &       \\
            \hline
            4 &   &   &   & $\Z^{40}$ &       \\
            \hline
            5 &   &   &   & $\Z^{32}$ & $\Z^{12}$     \\
            \hline
            6 &   &   &   &   & $\Z^{60}$     \\
            \hline
            7 &   &   &   &   & $\Z^{24}$  \\ \hline   
        \end{tabular}
    \end{tabular}
    \caption{$EMH_{k,\ell}
    (H)$}\label{fig:whitney1}
    \end{subfigure}%
    \begin{subfigure}{0.19\textwidth}
    \centering
 \begin{tikzpicture}
        \node[circle, draw, fill=black, scale=0.3] (1) at (0,1) {};
        \node[circle, draw, fill=black, scale=0.3] (2) at (-0.866,.5) {};
        \node[circle, draw, fill=black, scale=0.3] (3) at (0,0) {};
        \node[circle, draw, fill=black, scale=0.3] (4) at (-0.866,1.5) {};
        \node[circle, draw, fill=black, scale=0.3] (5) at (0.866,-0.5) {};       
        \draw (1) -- (2) -- (3) -- (1);
        \draw (1) -- (4);
        \draw[red] (3) -- (5);

        \node[circle, draw, fill=black, scale=0.3] (1) at (0,4) {};
        \node[circle, draw, fill=black, scale=0.3] (2) at (-0.866,3.5) {};
        \node[circle, draw, fill=black, scale=0.3] (3) at (0,3) {};
        \node[circle, draw, fill=black, scale=0.3] (4) at (-0.866,4.5) {};
        \node[circle, draw, fill=black, scale=0.3] (5) at (0.866,4.5) {};
            
        \draw (1) -- (2) -- (3) -- (1);
        \draw (1) -- (4);
        \draw[red] (1) -- (5);
    \end{tikzpicture}
    \caption{Graphs $H$ and $K$.}\label{fig:whitneyFL}
    \end{subfigure}
\begin{subfigure}{0.4\textwidth} 
 \begin{tabular}{c}
    \def\arraystretch{1.15}
        \begin{tabular}{|c|c|c|c|c|c|} \hline
        \def\arraystretch{1.25}
               \diagbox{$\ell$}{$k$} & 0 & 1 & 2 & 3 & 4  \\
            \hline
            0 & $\Z^5$ &   &   &   &      \\
            \hline
            1 &   &$\Z^{10}$&   &   &       \\
            \hline
            2 &   &   & $\Z^6$ &   &       \\
            \hline
            3 &   &   & $\Z^{16}$ &   &       \\
            \hline
            4 &   &   &   & $\Z^{28}$ &       \\
            \hline
            5 &   &   &   & $\Z^{16}$ &      \\
            \hline
            6 &   &   &   &$\Z^{4}$& $\Z^{28}$     \\
            \hline
            7 &   &   &   &$\Z^{2}$& $\Z^{36}$ \\
            \hline
            8 &   &   &   &    & $\Z^{14}$ \\ \hline
        \end{tabular}
         \end{tabular}
    \caption{$EMH_{k,\ell}(K)$}
    \label{fig:whitney2}
    \end{subfigure}
    \caption{Graphs $H$ and $K$ related by a Whitney flip (B) and their Eulerian Magnitude homology (A) and (C) respectively.}
\end{figure}
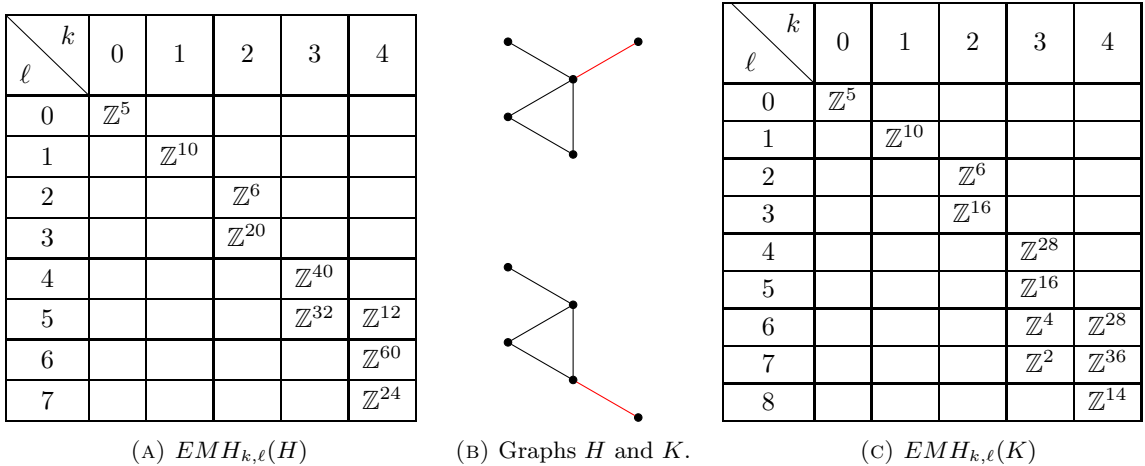

\section{Conjectures and future directions}\label{Conj}

Computational experiments using a program written in SageMath \cite{Martin_Magnitude_Homology} have revealed intriguing patterns in the relationship between magnitude homology and Eulerian magnitude homology. Specifically, our computations show that $MH_{4,5}(G_4^{(1\,2\,3)})$ 
contains torsion while $EMH_{4,5}(G_4^{(1\,2\,3)})$
 is torsion-free, providing an explicit example of a graph with torsion in its magnitude homology but not in its Eulerian magnitude homology for the same bi-degree.

Further computations for the family of graphs $G_k^\sigma$  suggest a potential pattern: torsion appears to persist along the second diagonal in magnitude homology. This raises the question of whether torsion in a specific bi-grading 
$(k,\ell)$  necessarily implies torsion along the entire diagonal for that bi-grading. However, our computations also show that for $\mathcal{G}(\mathbb{RP}^2)$
the graph of the face poset of a triangulation of  
$\mathbb{RP}^2$, $MH_{4,5}(\mathcal{G}(\mathbb{RP}^2))$  is torsion-free. This suggests that while some graphs exhibit persistent torsion along a diagonal, others—despite having torsion—do not follow this pattern.

These observations motivate several important directions for future research. A key open problem is to determine what structural properties of graphs influence the presence or persistence of torsion in their magnitude homology. Furthermore, we aim to further investigate the interplay between magnitude homology, Eulerian magnitude homology, and discriminant magnitude homology, seeking deeper connections between these invariants. Understanding how these different homology theories encode graph structure could provide valuable insights into their relationships and potential applications.

Based on the computations for the star trees $S_n$ that rely heavily on Corollary \ref{tree_corollary} we propose the following conjecture. 

\begin{conj}\label{strenghten_tree}
    For a any tree $T$, 
    $\text{rank}(DMH_{k,k}(T))=\text{rank}(MH_{k,k}(T))+\text{rank}(EMH_{k-1,k}(T))$
    for $k=3,4$ and $DMH_{k,k}(T)\cong MH_{k,k}(T)$ for $k=2$ and $k\geq 5$, while $DMH_{k,\ell}(T)=0$ for all $k\neq\ell$.
\end{conj}

While doing the computations for star trees, we also noticed an interesting recurrence relationship between the Eulerian magnitude homology groups for path trees $P_n=([n]_0,\{\{i,i+1\}_{i\in[n-1]_0}\})$. 

\begin{conj}
    For the path trees $P_n$, the Eulerian magnitude homology groups satisfy the following recurrence relation
    \begin{equation}\label{path_formula}
        \text{rank}(EMH_{k,\ell}(P_n))=(n-k+1)\text{rank}(EMH_{k,\ell}(P_k))
    \end{equation}
    for $0\leq k\leq n$ and for all $\ell\geq 0$.
\end{conj}

Thus, in order to classify the Eulerian magnitude homology groups of the path trees $P_n$, one only needs to come up with the initial conditions: i.e. find formulas for the ranks of the Eulerian magnitude homology groups $EMH_{n,\ell}(P_n)$ for $n\leq\ell\leq\ell_{\text{max}}$. It may even be useful to come up with a way to determine $\ell_{\text{max}}$ for $P_n$ given $n\geq 1$.

\bibliographystyle{plain}
\bibliography{refs}
\end{document}